\documentclass{amsart}
\usepackage{paralist}
\usepackage[pagewise]{lineno}

\usepackage[scale=0.63]{geometry}
\usepackage[numbered]{bookmark}
\usepackage{graphicx}
\usepackage{amsmath}
\usepackage{amsfonts}
\usepackage{amsthm} 
\usepackage{amssymb}
\usepackage{mathrsfs}
\usepackage{mathtools}
\usepackage{mathabx}
\usepackage{enumerate}
\usepackage{cite}
\usepackage{changepage} 
\usepackage{verbatim} 

\newtheorem{thm}{Theorem}[section]

\newtheorem{lem}[thm]{Lemma}
\newtheorem{prop}[thm]{Proposition}

\newtheorem{propx}{Proposition}

\DeclareMathOperator{\id}{id}

\newcommand{\R}{\mathbb{R}}
\newcommand{\Z}{\mathbb{Z}} 
\newcommand{\C}{\mathbb{C}} 
\newcommand{\T}{\mathbb{T}} 
\newcommand{\N}{\mathbb{N}}

\newcommand{\td}{\mathbb{T}^d}

\newcommand{\dom}{D_{r, s}}

\newcommand{\initialmod}{D_{r_0, s_0}}

\newcommand{\Avg}{\mathcal{M}_q}
\newcommand{\pw}{\partial_\omega}
\newcommand{\step}{5}
\newcommand{\hone}{\hspace{1cm}}
\newcommand{\dep}{d$, $l$, $r_0$, $s_0$, $\gamma$, $\tau$, $M_0$, $Q_0$, $h_0}
\newcommand{\depp}{d, l, r_0, s_0, \gamma, \tau, M_0, Q_0, h_0}
\newcommand{\deppp}{d, l, r_0, s_0, \gamma, \tau, M_0, h_0}
\newcommand{\firstop}{\mathcal{L}^1}
\newcommand{\secondop}{\mathcal{L}^2}
\newcommand{\thirdop}{\mathcal{L}^3}

\newcommand{\rsnorm}{2, r, s}
\newcommand{\rspnorm}{2, r^+, s^+}
\newcommand{\rsnnorm}{2, r_n, s_n}
\newcommand{\rsnpnorm}{2, r_{n +1}, s_{n +1}}
\newcommand{\polk}{\mathcal{T}^K}
\newcommand{\polkz}{\mathcal{T}^K_0}
\newcommand{\normalrs}{\mathcal{N}_{r, s}}

\newcommand{\normalrszero}{\mathcal{N}_{r_0, s_0}}
\newcommand{\normalrsn}{\mathcal{N}_{r_n, s_n}}
\newcommand{\normalrsnp}{\mathcal{N}_{r_{n + 1}, s_{n + 1}}}
\newcommand{\normalrsp}{\mathcal{N}_{r^+, s^+}}

\newcommand{\DCd}{\textup{DC}_d(\gamma, \tau)}

\begin{document}

\author{Frank Trujillo}
\date{}
\address{Institut für Mathematik, Universität Zürich, Winterthurerstrasse 190, CH-8057 Zürich, Switzerland}
\email{frank.trujillo@math.uzh.ch}

\title[Surviving Lower Dimensional Tori of a Resonant Torus]{Surviving Lower Dimensional Tori of an invariant Resonant Torus with any Number of Resonances}

\begin{abstract}
We provide sufficient conditions on integrable analytic Hamiltonians that guarantee the existence, under arbitrary sufficiently small analytic perturbations, of invariant lower dimensional tori associated to an invariant resonant torus of the unperturbed Hamiltonian. 
\end{abstract}

\maketitle

\section{Introduction} 
Let $H$ be a real analytic Hamiltonian over $\T^m \times \R^m$ of the form
\begin{equation}
\label{eq: initial_system}
 H(\theta, I) = N(I) + f(\theta, I).
 \end{equation}
Its associated \textit{Hamiltonian system} is given by
 \begin{equation}
 \label{eq: initial_equations}
 \left\{ \begin{array}{l}
 \dot{\theta} = \nabla N (I) + \partial_I f(\theta, I), \\
 \dot{I} = - \partial_\theta f(\theta, I), \\
 \end{array}\right.
 \end{equation}
and its solution, which we denote by $\Psi_H^t$, is called the \textit{Hamiltonian flow} of $H$. We refer to the coordinates $\theta$ and $I$ as \textit{angle} and \textit{action} variables and to $m$ as  the number of \textit{degrees of freedom} of the system.

For $f \equiv 0,$ the system above is \textit{integrable}, that is, its phase space is completely foliated by invariant tori whose restricted dynamics is given by (or more generally, smoothly conjugated to) translations. 
Indeed, if $f \equiv 0,$ the associated flow is given by
 \[ \Psi_H^t (\theta, I) = \left(R^t_{\nabla N(I)}(\theta), I\right) \text{ for all } t \in \R,\] 
 where for any $\omega \in \R^m$, $R^t_\omega$ denotes the continuous translation by $\omega$ on $\T^m$.
Hamiltonians of the form (\ref{eq: initial_system}) are sometimes called \textit{near-integrable} whenever one considers $f$ as a small perturbation of the integrable Hamiltonian $N$. 

A classical question in dynamical systems concerns the persistence of invariant structures for a given system after sufficiently small perturbations. In the context of near-integrable Hamiltonians, whenever 
 the Hessian matrix $D^2 N(I_0)$ is non-singular and the vector $\omega_0 = \nabla N(I_0)$ is \textit{Diophantine}, i.e. verifies a condition of the form 
\begin{equation*} 
\label{Diophantine}
 \left| \langle \omega_0 , k \rangle \right| \geq \dfrac{\gamma}{|k|^{\tau}} \hskip0.5cm \text{ for all } k \in \Z^m \,\setminus\, \{0\},
 \end{equation*}
for some $\gamma, \tau > 0$, the classical KAM theorem guarantees the \textit{persistence} of the invariant torus $T_{I_0} = \T^m \times \{ I_0 \}$ for sufficiently small perturbations of $N$, that is, provided $f$ is sufficiently small, the perturbed system $H$ admits an invariant 
torus, close to $T_{I_0}$, whose induced dynamics can be smoothly conjugated to $R^t_{\omega_0}$. 

On the other hand, if  $\omega_0$ is \textit{resonant}, i.e. if there exists $k \in \Z^m \, \setminus \, \{0\}$ such that $\langle \omega_0, k \rangle = 0,$ the classical KAM theorem cannot be applied. In this case, the \textit{resonant invariat torus} $T_{I_0}$ is completely foliated by invariant tori (for the unperturbed system) of positive codimension $l$ with respect to $T_{I_0}$,  whose restricted dynamics is conjugated to a translation by a non-resonant vector $\omega \in \R^{m - l}$.  The codimension $l$ is equal to the maximum number of linearly independent \textit{resonances} of $\omega_0$, where by resonance we mean any integer vector $k \in \Z^m \, \setminus \, \{0\}$  verifying $\langle \omega_0, k \rangle = 0.$ In this case, we say  that $\omega_0$ has exactly $l$ resonances.  

In this resonant setting, not only the hypotheses of the classical KAM theorem are not satisfied but, in general, the resonant invariant torus, or equivalently, the associated collection of \textit{lower dimensional invariant tori}, where use the term \textit{lower dimensional} to emphasize the fact that these invariant tori have dimension is smaller than the number of degrees of freedom of the system, does not persist under small perturbations. 

Nevertheless, invariant lower dimensional tori with the same dynamics as that of the ones in the invariant foliation of a resonant torus might still be found in the perturbed system.  Most of the existing results in this direction deal with \textit{generic} perturbations and hold for resonant vectors with any number of resonances \cite{treshchev_mechanism_1991}, \cite{eliasson_biasymptotic_1994} \cite{cheng_surviving_1999}.  However, similar results for \textit{arbitrary} perturbations are only available when the resonant vector has exactly $1$  or $m - 1$ resonances \cite{bernstein_birkhoff_1987}, \cite{cheng_birkhoff-kolmogorov-arnold-moser_1996}, \cite{cheng_surviving_1999}, \cite{plotnikov_kolmogorovs_2011}, \cite{corsi_lower-dimensional_2013}. 

The aim of this work is to fill this gap, by proving a positive result in the case of resonant tori with any number of resonances.
 
 \subsection{Perturbation of a resonant torus}
 \label{sc: example} 
 
 Let us start by illustrating the situation in the following example.  Fix $\omega_0 = (\omega, 0) \in \R^d \times \R^l = \R^m$, with $\omega \in \R^d$ non-resonant, and consider the integrable Hamiltonian
 \begin{equation}
 \label{eq: example}
 N(p, y) = \langle \omega, p \rangle - \frac{1}{2}|p|^2 + \frac{1}{2}|y|^2,
 \end{equation}
 and a perturbation of the form
  \[ f(q, x, p, y) = f(x),\]
in (\ref{eq: initial_system}), where we denote $(\theta, I)$  by  $(q, x, p, y) \in \T^d \times \T^l \times \R^d \times \R^l$. Then, for $H = N + f$, the system in (\ref{eq: initial_equations}) can be decoupled as
 \[  \left\{ \begin{array}{l}
 \dot{q} = \omega - p, \\
  \dot{p} = 0, \\
 \end{array}\right. \hskip1cm   \left\{ \begin{array}{l}
 \dot{x} =  y, \\
 \dot{y} = -\nabla f(x). \\
 \end{array}\right. \] 
For the unperturbed system, the torus $\T^{m} \times \{0\}$ is invariant and foliated by invariant tori of the form $\T^d \times \{(x, 0, 0)\}$, whose restricted dynamics are given by $R^t_\omega.$ For the perturbed system, any critical point $x_0 \in \T^l$ of $f$ defines an invariant torus of the form $\T^d \times \{(x_0, 0, 0)\}$, whose restricted dynamics is also given by $R^t_\omega.$ Moreover, if we suppose that $f$ admits a non-degenerate minimum at $x_0 \in \T^l$, the associated invariant lower dimensional invariant torus is of \textit{hyperbolic type} and therefore it is an isolated invariant torus. Using this fact and the particular form of the perturbation, it is easy to show that the perturbed system does not admit a collection of invariant tori as in the unperturbed system.

Given a Hamiltonian $H: \T^d \times \R^d \times \R^{l} \times \R^l \rightarrow \R$ of the form $H(q, x, p, y) = \langle \omega, p \rangle + O^2(x, p, y)$, we say that the invariant torus $\T^d \times \{(0, 0, 0)\}$ is of \textit{hyperbolic type} if, after a suitable symplectic change of coordinates, we can express $H$ as
\[ H(q, x, p, y) = \langle \omega, p \rangle + \frac{1}{2} \langle  M p, p \rangle + \langle \Omega x, y \rangle + O^3(x, p, y),\]
where $M$ is a symmetric, invertible matrix and $\Omega$ is a positive definite symmetric matrix.


Let us point out that the existence of lower dimensional invariant tori as in the previous example (although not necessarily of hyperbolic type) can be proven for general integrable Hamiltonians $N(p, y)$ and perturbations of the form $f(x, p, y)$. A proof of this fact can be found in \cite{cheng_birkhoff-kolmogorov-arnold-moser_1996}. In particular, this shows that for a general perturbation $f(q, x, p, y),$ the \textit{averaged perturbed system} $ N(p,y) + \Avg f(x, p, y)$, where
\[ \Avg f(x, p, y) = \int_{\T^d} f(q, x, p, y) dq,\]
admits such invariant lower dimensional tori. Averaged systems are fundamental importance in perturbation theory for near-integrable dynamical systems, since they often provide good approximations of the initial system. We refer the interested reader to \cite[Chapter 10]{arnold_mathematical_2007}.

\subsection{Previous results}
\label{sc: previous_results}

Although resonant invariant tori tend to disappear for general perturbations, the remnants of the associated foliations are not completely understood. As illustrated in the previous example, it is sometimes possible to find invariant lower dimensional tori, having the same dynamics as that of the tori belonging to the associated invariant foliation, in the perturbed system. 

Most of the existing results concerning the existence of such tori deal only with generic perturbations of a given integrable system and hold for resonant tori with any number of resonances.  See for example the works of D. Treshchev \cite{treshchev_mechanism_1991}, H. Eliasson \cite{eliasson_biasymptotic_1994} and C. Cheng, S. Wang \cite{cheng_surviving_1999}.  These results provide sufficient conditions on an integrable system $N$ and on the resonant vector $\omega_0 = \nabla N(0)$, which for simplicity we suppose of the form $\omega_0 = (\omega, 0) \in \R^d \times \R^l$, with $\omega$ non-resonant, such that for \textit{almost every} perturbation $f$, the following holds:   There exists $\epsilon_0 > 0$ such that for any $|\epsilon| < \epsilon_0$, the perturbed Hamiltonian $N + \epsilon f$ admits an invariant $d$-dimensional torus whose restricted dynamics is conjugated to $R^t_\omega$. Notice that the size of the allowed perturbation depends on the function $f$ being considered. 

These works rely heavily on results or techniques concerning the persistence of \textit{non-degenerate lower dimensional invariant tori}. By non-degenerate we mean that these tori admit a special set of coordinates around them, commonly classified as \textit{elliptic}, \textit{hyperbolic} or \textit{mixed}, see \cite{russmann_invariant_2001} for the concerning definitions and for an overview of several results in this setting. In contrast to the case of a resonant torus, these non-degenerate lower dimensional invariant tori are always isolated. However, it is possible to prove generic persistence results by exploiting properties of the perturbation in order to conjugate the system, in a neighbourhood of one of the lower dimensional invariant tori in the invariant foliation, to a Hamiltonian for which the existence of one of the aforementioned special set of coordinates is clear. 

Similar results for arbitrary perturbations are only available when the number of resonances $l$  is equal to $1$ or $m - 1$. Notice that in the latter case the resonant torus is completely foliated by periodic orbits. This situation was considered by D. Bernstein and A. Katok  in \cite{bernstein_birkhoff_1987}, where they proved that for any sufficiently small $C^2$ perturbation of a $C^2$ integrable \textit{convex} Hamiltonian, i.e. a Hamiltonian having a positive definite Hessian matrix, the perturbed system possesses at least $m$ periodic orbits.  

Concerning resonant tori with exactly one resonance,  C. Cheng \cite{cheng_birkhoff-kolmogorov-arnold-moser_1996} showed the existence of at least one invariant torus of codimension one, whose restricted dynamics are conjugated to $R^t_\omega$, for any sufficiently small analytic perturbation of an analytic convex Hamiltonian, where, as before and for simplicity, we suppose the resonant rotation vector $\omega_0$ to be of the form $\omega_0 = (\omega, 0) \in \R^d \times \R^l$, with $\omega$ non-resonant. Let us point out that in Cheng's result a \textit{relatively Diophantine condition} on the rotation vector $\omega_0$ of the resonant torus is required. In this simplified setting, this condition amounts to $\omega$ being Diophantine. We give a formal definition of this property in (\ref{eq: relatively_diophantine}). A similar result to that of C. Cheng, for a class of non-convex Hamiltonians, was announced by P. Plotnikov and I. Kuznetsov in \cite{plotnikov_kolmogorovs_2011}. A particular case of the previous result was proven by  L. Corsi, R. Feola and G. Gentile \cite{corsi_lower-dimensional_2013} for perturbations not depending on the action variable.  

\subsection{Invariant resonant tori with arbitrary number of resonances}

The main result of this work, Theorem \ref{thm: main_thm}, is a generalization of the results of  P. Plotnikov, I. Kuznetsov \cite{plotnikov_kolmogorovs_2011} and  L. Corsi, R. Feola, G. Gentile \cite{corsi_lower-dimensional_2013} to invariant resonant tori with any number of resonances $1 \leq l \leq m - 1$. Let us mention that the method we use to prove Theorem \ref{thm: main_thm}, a KAM scheme with counter-term, is different from the methods in the works mentioned above. 

We will consider a class of non-convex analytic integrable Hamiltonians with $m$ degrees of freedom for which $\T^m \times \{0\}$ is a resonant invariant torus with exactly $l$ resonances, and show that, for any Hamiltonian $N$ in this class such that the resonant rotation vector $\omega_0 = \nabla N (0)$ verifies a relative Diophantine condition of the form
\begin{equation}
\label{eq: relatively_diophantine}
\exists K \in SL(m, \Z) \text{ such that } K\omega_0 = (\omega, 0) \in \R^d \times \R^l \text{ and } \omega \text{ is Diophantine,}
\end{equation}
 then any sufficiently small analytic perturbation of $N$ possesses an invariant $d$-dimensional torus whose restricted dynamics is conjugated to $R_\omega^t$.

Notice that for resonant vector $\omega_0$ with exactly $l$ resonances, a matrix $K \in SL(m, \Z) $ taking it to the form $K\omega_0 = (\omega, 0)$, with $\omega$ non-resonant, always exists. To see this it suffices to take $l$ linearly independent primitive resonances $k_{d + 1}, \dots, k_{d +l} \in \Z^m$ and to construct vectors $k_1, \dots, k_d \in  \Z^m$ such that the matrix $K$ of rows $k_1, \dots, k_{d + l}$ is unimodular. 

We point out that for $l = 1$, the class of non-convex Hamiltonians in Theorem \ref{thm: main_thm} coincides with the one considered in \cite{plotnikov_kolmogorovs_2011} and contains the class of Hamiltonians studied in \cite{corsi_lower-dimensional_2013}. The integrable Hamiltonian (\ref{eq: example}) considered in Section \ref{sc: example} belongs to this class. 

For the sake of clarity we postpone precise statements of our results to the next section.  

\section{Statements of the main results}

In this section we will provide precise statements for our main results. We start by setting some general notations that will be used throughout this work.

\subsection{Notations}
Given $z \in \C$ we denote its modulus by $|z|$. For $z \in \C^m$, we denote 
\[ |z|_1 = |z_1| + \dots + |z_m|, \hone |z| = \sqrt{|z_1|^2 + \dots + |z_m|^2}.\]
Let $ M_n(\R)$ denote the set of square matrices of order $n$ with values in $\R$. For $M \in M_n(\R)$, we denote
 \[ \nu_{\max}(M) = \max_{v \in \mathbb{S}^{n - 1}} \langle M v, v \rangle, \hskip0.5cm \nu_{\min}(M) = \min_{v \in \mathbb{S}^{n - 1}} \langle M v, v \rangle, \hskip0.5cm \| M\| = \max_{v \in \mathbb{S}^{n - 1}} | Mv|. \]
Notice that for a symmetric matrix $M$ the values $\nu_{\max}$, $\nu_{\min}$ correspond to the biggest and smallest eigenvalues of $M$ respectively. We denote the identity matrix by $I_n \in M_n(\R)$. Given $\gamma, \tau > 0$, we say that $\omega \in \R^d$ is \textit{Diophantine of type} $(\gamma, \tau)$ if 
\[ | \langle \omega, k \rangle | \geq \dfrac{\gamma}{|k|^{d + \tau}} \hskip0.5cm \text{ for all } k \in \Z^d \,\setminus\, \{ 0 \}.\] 
We denote by $\DCd$ the set of Diophantine vectors of type $(\gamma, \tau)$ in $\R^d$.  Given $k \in \N \cup \{ \infty, \omega\}$, we denote by $C^k(U , V)$ the space of $C^k$ functions defined on $U$ and taking values in $V$. Unless otherwise specified, all the analytic functions we consider are supposed to be real analytic. If $V = \C$, we denote this space simply by $C^k(U)$. Similarly, given $k_1, k_2 \in \N   \cup \{ \infty, \omega\}$, we denote by $C^{k_1, k_2}(U_1 \times U_2, V)$ and $C^{k_1, k_2}(U_1 \times U_2)$ the space of functions defined on $U_1 \times U_2$ which are of class $C^{k_1}$ in the first coordinate and of class $C^{k_2}$ in the second one. Given $f: U \subset \C^m \rightarrow \C^n$ we denote its sup-norm by 
\[ \| f\|_U = \sup_{z \in U} | f(z) |.\]
Given $f \in C^k(U, \C^n)$, with $k \in \N$ and $U \subset \C^m$, we denote its $C^k$-norm by
\[ \| f\|_{C^k(U)} = \sum_{\substack{\alpha \in \N^m \\ |\alpha|_1 \leq k}} \|\partial^\alpha f\|_U.\] 
If there is no risk of confusion we will denote $\| \cdot\|_{C^k(U)}$ simply by $\| \cdot\|_{C^k}$. Given $f \in C^{k_1, k_2}(U_1 \times U_2, \C^n)$, with $k_1, k_2 \in \N$ and $U_1 \subset \C^d$, $U_2 \subset \C^l$, we define its $C^{k_1, k_2}$-norm by
\[ \| f\|_{C^{k_1, k_2}(U_1 \times U_2)} = \sum_{\substack{\alpha = (\alpha_1, \alpha_2) \in \N^d \times \N^l \\ |\alpha_1|_1 \leq k_1,  |\alpha_2|_1 \leq k_2 }} \|\partial^\alpha f\|_{U_1 \times U_2}.\] 
As before, if there is no risk of confusion, we will denote $\| \cdot\|_{C^{k_1, k_2}(U_1 \times U_2)}$ simply by $\| \cdot\|_{C^{k_1, k_2}}$.

 Given a Hamiltonian $H: M \rightarrow \R$ of class $C^1$ over a symplectic manifold $(M, \omega)$,  we denote its associated Hamiltonian vector field by $X_H$ and the corresponding flow by $\Psi_H^t$. We denote by $\{ \cdot, \cdot \}$ the Poisson bracket associated to $\omega$. 
 The formal definition of these objects can be found in the Appendix. 

\subsection{Statements}

Let
 \[ \T^m_r = \left( \{ z \in \C \mid |\textup{Im}(z)| < r\} / \Z \right)^m, \hone B^{m}_s = \{ z \in \C^m \, \mid \, |z| < s \}, \] 
and define
\[\Sigma^{d, l}_{r, s} = \T^{d}_r \times \T^{l}_r \times B^{d}_s \times B^{l}_s, \hone D^{d, l}_{r,s} = \T^d_r \times B^{l}_s \times B^{d}_s \times B^{l}_s.\]
We denote coordinates in $\Sigma^{d, l}_{r, s}$ or $D^{d, l}_{r, s}$ by $(q, x, p, y)$. Notice that these domains define symplectic manifolds when endowed with the canonical symplectic form $$\sum_{i= 1}^{d} dq_i \wedge dp_i + \sum_{i = 1}^l dx_i \wedge dy_i.$$ The following is a simplified version of our main result (Theorem \ref{thm: main_thm}). 

\begin{thm}
\label{thm: main_thm_simplified}
Suppose $r, s, \gamma, \tau > 0$, $d, l \in \N$, $\omega \in DC_d(\gamma, \tau)$.  Let $M \in M_d(\R)$ and $Q \in M_l(\R)$ be, respectively, negative and positive definite symmetric matrices. Define $N \in C^\omega(\Sigma^{d, l}_{r, s})$ as
\begin{equation}
\label{eq: initial_unperturbed_simplified}
 N(p, y) =  \langle \omega, p \rangle + \dfrac{1}{2}\langle M p , p \rangle + \dfrac{1}{2}\langle Q y , y \rangle. 
 \end{equation}
There exists $\epsilon_0(r, s, d, l, \gamma, \tau, M, Q) > 0$, such that for any $f \in C^\omega(\Sigma^{d, l}_{r, s})$ with $\| f \|_{\Sigma^{d, l}_{r, s}} < \epsilon_0$, the Hamiltonian $H \in C^\omega(\Sigma^{d, l}_{r, s})$ given by
\[ H(q, x, p, y) = N(p, y) + f(q, x, p, y),\]
admits an invariant $d$-dimensional torus parametrized by a function $\phi \in C^\omega(\td, \Sigma_{r, s}^{d, l})$ obeying $\phi^* X_H = X_\omega,$ where $X_\omega$ denotes the constant vector field $\omega$ over $\T^d$.
\end{thm}

In the following, we will introduce a parameter $\varphi \in \T^l$ and consider, instead of $\Sigma^{d, l}_{r, s}$, domains of the form $\T^l \times D^{d, l}_{r, s}$. We do this in order to localize a perturbed Hamiltonian  $H = N + f \in C^\omega(\Sigma^{d, l}_{r, s})$ around all tori of the form $\T^d \times \{(\varphi, 0, 0)\}$ (which by assumption are invariant for the unperturbed system $N$) simultaneously. For $H$ as in the previous statement, this amounts to consider a parametrized Hamiltonian over $\T^l \times D^{d, l}_{r, s}$ of the form $N(p, y) + f(q, x + \varphi, p, y)$.

Given $f \in C(\T^l \times D^{d, l}_{r, s})$  we denote by $\Avg f$ the averaged function
\[ \Avg f(\varphi, x, p, y) = \int_{\td} f(\varphi, q, x, p, y)dq.\] 
The next theorem can be seen as a parametrized version of Theorem \ref{thm: main_thm_simplified}.  Its proof is based on a counter-term KAM scheme which we summarize in Proposition \ref{prop: iteration}. 

\begin{thm}
\label{thm: main_thm_parametrized}
Suppose $r_0, s_0, \gamma, \tau > 0$, $d, l \in \N$, $\omega \in DC_d(\gamma, \tau)$.  Let $M_0 \in M_d(\R)$ and $Q_0 \in M_l(\R)$ be, respectively, negative and positive definite symmetric matrices. Let $h_0 \in C^\omega(\T^l \times D^{d, l}_{r, s}) \cap O^3(p, y).$  Define $N_0 \in C^\omega(\T^l \times D^{d, l}_{r, s})$ as 
\begin{equation}
\label{eq: initial_unperturbed}
 N_0(\varphi, q, x, p, y) =  \langle \omega, p \rangle + \dfrac{1}{2}\langle M_0 p , p \rangle + \dfrac{1}{2}\langle Q_0 y , y \rangle + h_0(\varphi, q, x, p, y). 
 \end{equation}
There exists $\epsilon_0(d, l, r_0, s_0, \gamma, \tau, M_0, Q_0, h_0) > 0$, such that for any $f_0 \in C^{\omega}(\T^l \times D_{r_0, s_0}^{d, l})$ with $\| f_0\|_{\T^l \times D_{r_0, s_0}^{d, l}} < \epsilon_0$ and obeying 
\begin{equation}
\label{eq: equal_derivatives}
\Avg(\partial_ xf_0) = \Avg(\partial_\varphi f_0),
\end{equation}
there exists $\varphi_0 \in \T^l$ for which the Hamiltonian $\overline{H} \in C^\omega(D_{r_0, s_0}^{d, l})$ given by 
\[ \overline{H}(q, x, p, y)  = N_0(\varphi_0, q, x, p, y) + f_0(\varphi_0, q, x, p, y), \]
admits an invariant $d$-dimensional torus parametrized by an embedding $\overline{\phi} \in C^\omega(\td, D_{r_0, s_0}^{d, l})$ obeying $\overline{\phi}^* X_{\overline{H}} = X_\omega,$ where $X_\omega$ denotes the constant vector field $\omega$ over $\T^d$. Furthermore
\[ \left\| \overline{\phi} - \phi_0\right\|_{\T^d} = O\big( \| f_0 \|_{C^2}^{1/2}\big),\] where $\phi_0 :\td \rightarrow D_{r_0, s_0}^{d, l}$ is the trivial embedding given by $q \mapsto (q, 0)$.
\end{thm} 

Theorem \ref{thm: main_thm_simplified} can be generalized to a wider class of Hamiltonians, not necessarily of the form (\ref{eq: initial_unperturbed_simplified}). In fact, Theorem \ref{thm: main_thm_parametrized} implies the following.

\begin{thm}
\label{thm: main_thm}
Let $r, s > 0$, $m \in \N$, $h \in C^\omega(\Sigma_{r, s}^{d, l}) \cap O^3(p, y)$ and $N \in C^\omega(B_{s}^m)$ with $\omega_0 = \nabla N(0)$ having exactly $1 \leq l \leq m -1$ resonances. Denote $d = m - l.$ Suppose $K \in SL(m, \Z)$ is such that $K \omega_0 = (\omega, 0) \in \R^d \times \R^l$ and denote
\[ K \nabla^2 N(0) K^T = \left( \begin{array}{@{}c|c@{}}
A_{d \times d} & B_{d \times l} \\
\hline
B^T_{l \times d} & C_{l \times l} 
\end{array}\right).\] 
Assume that the following conditions hold:
\begin{enumerate}[(i)]
\item $\omega \in DC_d(\gamma, \tau)$ for some $\gamma, \tau > 0$,
\item $\nabla^2 N (0)$ and $C$ are non singular,
\item $A + B^T C^{-1} B$ and $C$ are, respectively, positive and negative definite. 
\end{enumerate}
Then, there exists $\epsilon_0(d, l, r, s, \gamma, \tau, N, K, h) > 0$, such that for any $f \in C^\omega(\Sigma_{r, s}^{d, l})$ obeying $\| f \|_{\Sigma_{r, s}^{d, l}} < \epsilon_0$, the Hamiltonian $H \in C^\omega(\Sigma_{r, s}^{d, l})$ given by 
\[H(q, x, p, y) = N(p, y) + h(q, x, p, y) + f(q, x, p, y),\]
admits an invariant $d$-dimensional torus parametrized by a function $\phi \in C^\omega(\td,\Sigma_{r, s}^{d, l})$ obeying $\phi^* X_H = X_\omega,$ where $X_\omega$ denotes the constant vector field $\omega$ over $\T^d$.
\end{thm}

\begin{proof}
Let $\kappa = \min\{\|K\|, \| K\|^{-1}\}$, $r' = \kappa r$ and $s' = \kappa s$. Denote by $\Psi$ the symplectic change of coordinates given by
\[
\begin{array}{cccc}
\Psi : & \Sigma_{r', s'}^{d, l} & \rightarrow & \Sigma_{r, s}^{d, l} \\
& (q, x, p, y) & \rightarrow & (K^{-1}(q, x), K^T (p, y)).
\end{array}
\]
Then
\begin{equation*}
\label{eq: first_reduction}
\begin{aligned}
H \circ \Psi (q, x, p, y) & = \langle \omega, p \rangle + \frac{1}{2}\langle Ap, p \rangle + \frac{1}{2}\langle Cy, y \rangle + \langle Bp, y\rangle\\
& \quad + f(K^{-1}(q, x), K^T(p, y)) + O^3(p, y).
\end{aligned}
 \end{equation*}
Consider the parametrized Hamiltonian $H' \in C^\omega(\T^l \times D_{r', s'}^{d, l})$ given by
\begin{equation*}
\label{eq: initial_Hamiltonian}
H'(\varphi, q, x, p, y) = H \circ \Psi (q, x + \varphi, p, y).
\end{equation*}
Notice that for $\varphi_0 \in \T^l$ fixed, an invariant torus for $\overline{H} = H_0(\varphi_0, \cdot) \in C^\omega(D_{r', s'}^{d, l})$ uniquely defines an invariant torus for $H \circ \Psi$. In fact, if $\overline{\phi} : \T^d \rightarrow D_{r', s'}^{d, l}$ parametrizes an invariant torus for $\overline{H}$, then $\phi : \T^d \rightarrow \Sigma_{r', s'}^{d, l}$ given by 
\[\phi(q) = (0, \varphi_0, 0, 0) + \overline{\phi}(q),\]
parametrizes an invariant torus for $H \circ \Psi$. Let $\kappa' = \min\{\frac{1}{2}, \|C^{-1}B\|\}$, $r_0 = \kappa' r'$ and $s_0 = \kappa' s'$. Define
\[
\begin{array}{cccc}
\Phi : &\T^l \times D_{r_0, s_0}^{d, l} & \rightarrow &\T^l \times D_{r', s'}^{d, l} \\
& (\varphi, q, x, p, y) & \rightarrow & (\varphi, q + (C^{-1}B)^Tx, x, p, y - C^{-1}Bp).
\end{array}
\]
Notice that for $\varphi \in \T^l$ fixed, $\Phi$ defines a symplectic transformation from $D_{r_0, s_0}^{d,l}$ to $D_{r', s'}^{d, l}$. Denoting $Q_0 = A + B^T C^{-1}B,$  $M_0 = C,$ $f' = f \circ \Psi \in C^\omega(\Sigma_{r', s'}^{d, l})$ and 
 \[ h_0(\varphi, q, x, p, y) = h(K^{-1}(q + (C^{-1}B)^Tx, x + \varphi), K^T(p, y - C^{-1}Bp)),\]
 \[ f_0(\varphi, q, x, p, y) = f'(q + (C^{-1}B)^Tx, x + \varphi, p, y - C^{-1}Bp),\]
the parametrized Hamiltonian $H_0 = H' \circ \Phi \in C^\omega(\T^l \times D_{r_0, s_0}^{d, l})$ is given by 
\begin{equation*}
 \begin{aligned}
 \label{eq: second_reduction}
 H_0 (\varphi, q, x, p, y) & = \langle \omega, p \rangle + \frac{1}{2}\langle  M_0p, p \rangle + \frac{1}{2}\langle Q_0y, y \rangle + h_0(\varphi, q, x, p, y) + f_0(\varphi, q, x, p, y),
 \end{aligned}
 \end{equation*}
with $h_0 = O^3(p, y)$. Noticing that $f_0$ satisfies (\ref{eq: equal_derivatives}), the result now follows by Theorem \ref{thm: main_thm_parametrized}.
\end{proof}

The remaining of the paper is devoted to the proof of Theorem \ref{thm: main_thm_parametrized}.

\subsection{Outline of the proof}
In this section we sketch the proof of Theorem \ref{thm: main_thm_parametrized} which is based on a counter-term KAM scheme. The iterative procedure proposed here is formally summarized in Proposition \ref{prop: iteration} and uses a \textit{normal form} whose formal definition we postpone to Section \ref{sc: normal_form}. For the moment, we can think of the normal form as a family of parametrized analytic Hamiltonians of the form $N: \T^l \times D_{r, s}^{d,l} \rightarrow \C$ (where $\varphi \in \T^l$ denotes the parameter) for which the existence of the desired invariant torus can be easily deduced for at least one of the parameters. 

Let us start by giving a brief sketch of the iterative scheme. Broadly speaking, we will construct sequences of real numbers $r_n \searrow r_\infty > 0$,  $s_n \searrow s_\infty > 0$, which we will use to define domains $D_{r_{n}, s_{n}}^{d,l}$, and sequences $\alpha_n,$ $\Phi^n,$ $N_n,$ $f_n$ (all depending on the parameter $\varphi \in \T^l$) corresponding, respectively, to the function $\alpha_n: \T^l \rightarrow \R$ characterizing the counter-term $\langle \alpha_n(\varphi), x \rangle,$ to a family of symplectic coordinates $\Phi^n: \T^l \times D_{r_{n - 1}, s_{n - 1}}^{d,l} \rightarrow D_{r_{n}, s_{n}}^{d,l},$ to a parametrized Hamiltonian in normal form $ N_n: \T^l \times D_{r_{n}, s_{n}}^{d,l} \rightarrow \C$ and to the error term $ f_n: \T^l \times D_{r_{n}, s_{n}}^{d,l} \rightarrow \C$, at the $n$-th step of the iterative procedure. These sequences will obey
 \begin{equation}
 \label{eq: n_step}
 (N_0 + f_0 - \langle \alpha_n, x \rangle ) \circ \Phi^n = N_n + f_n 
 \end{equation}
and their limits (in the $C^{2}$-norm) when $n$ goes to infinity will be well defined, with $f_n$ converging to zero and $N_n$ converging to a well-defined parametrized Hamiltonian $N_\infty$ in normal form. Intuitively, at each step of the iterative procedure the Hamiltonian in the LHS of (\ref{eq: n_step}) is getting closer to a parametrized Hamiltonian in normal form. By making $n$ go to infinity in (\ref{eq: n_step}) the iterative procedure just described will yield to
\[ (N_0 + f_0 - \langle \alpha_\infty, x \rangle ) \circ \Phi^\infty = N_\infty, \]
for some $\alpha_\infty \in C^{2}(\T^l,\R^l)$, some family of symplectic transformations $\Phi^\infty \in C^{2, \omega}(\T^l \times D_{r_\infty, s_\infty}^{d,l}, D_{r_0, s_0}^{d,l})$ and some $N_\infty \in C^{2, \omega}(\T^l \times D_{r_\infty, s_\infty}^{d,l}) $ in normal form. Assuming the existence of these functions, Theorem \ref{thm: main_thm_parametrized} will be proved if there exists $\varphi_0$ in $\T^l$ so that, simultaneously, $\alpha_\infty(\varphi_0) = 0$ and the Hamiltonian $N_\infty(\varphi_0, \cdot)$ admits an invariant torus as desired. 

As we shall see, the normal form $N_\infty$ and the counter-term $\alpha_\infty$ are closely related. In fact, if we denote by $\epsilon_n$ the $C^2$-norm of $f_n$ at the $n$-th step of the iterative procedure, the counter-term $\alpha_n$, when restricted to the set of $\varphi$ for which $N_n(\varphi, \cdot)$ admits an invariant torus as desired, will be $\epsilon_n$-close to the gradient of a smooth function $\zeta_n$. This is the content of Proposition \ref{prop: vanishing_counter_term}. Furthermore, the sequence $\zeta_n$ will converge in the $C^2$-topology to a well defined function $\zeta_\infty$ with the following property: If $\varphi_0 \in \T^l$ is a local maxima for $\zeta_\infty$, then for all $n \in \N$ the Hamiltonian $N_n(\varphi_0, \cdot)$ admits an invariant torus as desired. From this we will conclude that $\alpha_\infty(\varphi_0) = 0$. 

The iterative KAM scheme just described is summarized in Proposition \ref{prop: iteration} and proven in Section \ref{sc: proof_A}. The existence of $\varphi_0 \in \T^l$ for which the normal form $N_\infty(\varphi_0, \cdot)$ possesses an invariant torus and the counter-term vanishes is the content of Proposition \ref{prop: vanishing_counter_term}, which we prove in Section \ref{sc: proof_B}. 

\subsection{Normal Form}
\label{sc: normal_form}
 Two things are mainly sought in the definition of the normal form. First, an invariant torus with the desired properties must exist for the induced flow. Secondly, the associated cohomological equation (see Lemma \ref{lem: cohomological_equation}), which arises naturally when trying to conjugate a perturbed normal form to a Hamiltonian in normal form, must be solvable. Let us introduce some notations that will be useful in its definition. 
 
For $d, l \in \N^\ast$ fixed and for any $r, s > 0$ we denote by $\normalrs$ the vector space
\[
\begin{aligned}
\normalrs & = \R^d \times \C^\infty(\T^l) \times C^\infty(\T^l,M_l(\R)) \times C^\infty(\T^l,M_{l \times d }(\R)) \\
 & \quad \times C^\infty(\T^l,M_{d}(\R))  \times C^\infty(\T^l,M_{l}(\R))  \times C^{\infty, \omega}(\T^l \times D_{r,s}) \\
 & \quad \times C^{\infty, \omega}(\T^l \times D_{r,s}) \cap O^3(x, p, y).
 \end{aligned}
 \]
We denote elements of $\normalrs$ as tuples $\mathbf{N} = (w, c, \beta, \Gamma, M, Q, g, h)$ and its coordinates by $w(\mathbf{N}), c(\mathbf{N}), \dots, h(\mathbf{N})$. We endow $\normalrs$ with the norm
\[ \| \mathbf{N} \|_{\mathcal{N}_{r, s}} = \max\left\{|w|, \|c\|_{C^2}, \|\beta\|_{C^2}, \|\Gamma\|_{C^2},  \|M\|_{C^2}, \|Q\|_{C^2}, \| g\|_{C^2}, \| h\|_{C^2} \right\}.\]
 To each $\mathbf{N} \in \normalrs$, we associate a parametrized Hamiltonian by means of the linear operator
$$T_{r,s} :\normalrs \rightarrow C^{\infty, \omega}(\T^l \times D_{r,s}),$$
given by
\begin{equation}
 \label{eq: normal_form_hamiltonian}
 \begin{aligned}
T_{r,s}(\mathbf{N})& (\varphi, q, x, p, y) = c(\varphi) + \langle w, p \rangle + \dfrac{1}{2}\langle M(\varphi)p , p \rangle \\
 & + \dfrac{1}{2}\langle Q(\varphi)y, y\rangle + \langle \Gamma(\varphi) p, x \rangle + \dfrac{1}{2}\langle \beta(\varphi) x , x \rangle  \\
 & + g(\varphi, q, x, p, y) + h(\varphi, q, x, p, y).
 \end{aligned}
\end{equation}
If there is no risk of confusion, for $\mathbf{N} \in \normalrs$ we denote $T_{r,s}(\mathbf{N})$ simply by $N$. As an abuse of notation we refer to elements of $\normalrs$ indistinctly as tuples or functions, where the function associated is given by the linear operator $T_{r,s}$. 

Given $v \in \R^d$, $\delta \geq 0$,  we say that $\mathbf{N} \in \normalrs$ is in \textit{$(v, \delta)$-normal form} if $w(\mathbf{N}) = v$ and $g(\mathbf{N})(\varphi, \cdot) = 0 = \partial_\varphi g(\mathbf{N})(\varphi, \cdot),$ for all $\varphi \in \T^l$ satisfying $\nu_{\max}(\beta(\mathbf{N})(\varphi)) \leq \delta$.
Let
\[\normalrs^{v, \delta} = \left\{ \mathbf{N} \in\normalrs \, \mid \, N \text{ is in } (v, \delta)\text{-normal form} \right\}. \]
Notice that given $\mathbf{N}$ in $(\omega, \delta)$-normal form, and for any $\varphi \in \T^l$ such that 
\begin{equation}
\label{eq: existence_torus}
\nu_{\max}(\beta(\mathbf{N})(\varphi)) \leq \delta,
\end{equation}
the Hamiltonian $N_\varphi: D_{r, s}^{d, l} \rightarrow \C$ given by $N_\varphi = N(\varphi, \cdot)$ satisfies $$X_{N_\varphi}(q, 0, 0, 0) = (\omega, 0, 0, 0)^T.$$
Therefore, the associated hamiltonian flow $\Psi^t_{N_\varphi}$ possesses an invariant $d$-dimensional torus with rotation vector $\omega$. As we will see in Section \ref{sc: proof_A}, given $\mathbf{N} \in \normalrs$ in $(\omega, \delta)$-normal form and for any $f: \T^l \times D_{r, s}^{d, l} \rightarrow \C$ with sufficiently small $C^2$-norm, the associated cohomological equation, (\ref{eq: cohomological_equation}) in Lemma \ref{lem: cohomological_equation}, which appears naturally when trying to conjugate the Hamiltonian $N + f$ to a Hamiltonian in $(\omega, \delta_+)$-normal form $\big( \text{up to an error term of order } \| f\|_{C^2}^{3/2}\big),$ can be solved for some $0 < \delta_+ < \delta.$

\section{Proof of Theorem \ref{thm: main_thm_parametrized}}

For the remaining of this work we fix $r_0, s_0, \gamma, \tau, d, l, \omega, M_0, Q_0, h_0$ as in Theorem \ref{thm: main_thm_parametrized}. Namely, we fix
\[ d, l \in \N^*, \hskip0.5cm r_0, s_0, \gamma, \tau > 0, \hskip0.5cm \omega \in DC_d(\gamma, \tau),\]
\[Q_0 \in M_l(\R), \hskip0.5cm M_0 \in M_{d}(\R), \hskip0.5cm h_0 \in C^{\omega}(\T^l \times D^{d, l}_{r_0,s_0}) \cap O^3(p, y),\]
with $M_0$  and $Q_0$, respectively, negative and positive definite symmetric matrices.   Since $d, l$ are fixed, we denote $m = d + l$ and write simply $D_{r, s}, \Sigma_{r, s}$ instead of $D^{d, l}_{r, s}, \Sigma^{d, l}_{r, s}$. Notice that, up to consider the symplectic change of coordinates 
\[ (q, x, p, y) \mapsto (q, S^{-1}x, p, S^Ty),\]
where $S \in GL_l(\R)$ obeys $SQ_0S^T = I_l$, we can suppose WLOG that the matrix $Q_0$ in Theorem \ref{thm: main_thm_parametrized} is the identity matrix. We denote by $N_0 \in C^\omega(\T^l \times D_{r_0, s_0})$ the parametrized Hamiltonian given by (\ref{eq: initial_unperturbed}) when $Q_0 = I_l$, namely
\begin{equation}
\label{eq: initial_simplified}
N_0(\varphi, q, x, p, y) = \langle \omega, p \rangle + \dfrac{1}{2}\langle M_0 p , p \rangle + \dfrac{1}{2} |y|^2 + h_0(\varphi, q, x, p, y).
\end{equation}
Notice that using the notations introduced in Section \ref{sc: normal_form}, $N_0$ corresponds to the Hamiltonian associated to the normal form
 \begin{equation}
 \label{eq: initial_unperturbed_nf}
 \mathbf{N}_0 = (\omega, 0, 0, 0, M_0, I_l, 0, h_0) \in \normalrszero^{\omega, \delta},
 \end{equation}
 for any $\delta \geq 0$.  All of the notations above will be used freely in the following sections.

\subsection{The KAM Scheme}

Let us state the iterative KAM scheme that will yield to Theorem \ref{thm: main_thm_parametrized}. In the following, for the sake of simplicity, given functions defined on domains of the form $D_{r, s}$ or $\T^l \times D_{r, s}$, we denote their $C^2$ norm on these domains simply by $\| \cdot \|_{2, r, s}$.
 
\begin{propx}
\label{prop: iteration}
There exist sequences $r_n \searrow r_\infty > \frac{r_0}{2},$ $s_n \searrow s_\infty > \frac{s_0}{2},$ $\delta_n \searrow 0,$
depending only on $r_0, s_0$ and positive constants $\epsilon, C$ depending only on $\deppp,$ such that for all $n \geq 1$ and for any $f_0 \in C^{\infty, \omega}(\T^l \times \initialmod)$ with $ \| f_0\|_{2, r_0, s_0} < \epsilon$ and obeying 
\[ \Avg(\partial_x f_0) = \Avg(\partial_\varphi f_0),\]
there exist sequences $\mathbf{N}_n \in \mathcal{N}^{\omega, \delta_n}_{r_n, s_n},$ with $Q(\mathbf{N}_n) = I_l,$  $\alpha_n \in C^\infty(\T^l, \R^l),$ $f_n \in C^{\infty, \omega}(\T^l \times D_{r_n, s_n})$ and $\Phi^n \in C^{\infty, \omega}(\T^l \times D_{r_n,s_n}, D_{r,s}),$ with $\Phi^n$ symplectic for $\varphi \in \T^l$ fixed, obeying
 \begin{gather*}
 \max\left\{ \| \alpha_{n + 1} - \alpha_{n} \|_{C^2(\T^l)} ,   \| \mathbf{N}_{n + 1} - \mathbf{N}_{n}\|_{\normalrsnp}  \right\} < \epsilon_{n}^{\frac{1}{2}}, \\
 \max \Big\{  \| f_{n}\|_{\rsnnorm},  \|\Avg\Phi^{n}_{x}(\varphi, 0)\|_{C^2(\T^l)} \Big\} < \epsilon_n,\\
 \| \Phi^{n + 1} - \Phi^{n}\|_{\rsnpnorm} < C \epsilon_n^{\frac{1}{2}},
 \end{gather*}
for all $n \in \N$, where $\alpha_0 = 0,$ $\Phi^0 = \id_{D_{r_0, s_0}},$ $\epsilon_{n} = \epsilon ^{\left(\frac{3}{2}\right)^n},$
and such that
\begin{equation}
\label{step_conjugacy}
 (N_0 + f_0 - \langle \alpha_n, x \rangle ) \circ \Phi^n = N_n + f_n,
 \end{equation}
with $\mathbf{N}_0$ as in (\ref{eq: initial_unperturbed_nf}). 
In particular, there exist $\mathbf{N}_{\infty} \in \mathcal{N}^{\omega, 0}_{r_0 / 2, s_0 / 2},$ $\alpha_{\infty} \in C^2(\T^l, \R^l)$ and $\Phi^{\infty} \in C^{2, \omega}(\T^l \times D_{r_0 / 2, s_0 / 2}, \initialmod),$
with $\Phi^{\infty}$ symplectic for $\varphi \in \T^l$ fixed, such that 
\begin{equation}
\label{eq: final_conjugacy}
(N_0 + f_0 - \langle \alpha_{\infty}, x \rangle ) \circ \Phi^{\infty} = N_{\infty}.
\end{equation}
\end{propx}

\subsection{Vanishing of the counter-term}
As we saw in Section \ref{sc: normal_form}, a parametrized Hamiltonian in $(\omega, \delta) $-normal form $\mathbf{N} \in \normalrs^{\omega, \delta}$ admits a $d$-dimensional invariant torus with rotation vector $\omega$ for the restricted Hamiltonian $N(\varphi_0, \cdot)$ provided that 
\[ \nu_{\max}(\beta(\mathbf{N})(\varphi_0)) \leq \delta.\]
Thus, assuming Proposition \ref{prop: iteration}, the existence of an invariant torus as in Theorem \ref{thm: main_thm_parametrized} for the Hamiltonian $N_0 + f_0$ will be a consequence of the following.
\begin{propx}
\label{prop: vanishing_counter_term}
Let $\mathbf{N}_\infty, \alpha_\infty$ as in Proposition \ref{prop: iteration}. There exists $\varphi_0 \in \T^l$ such that 
\begin{equation}
\label{eq: existence_condition}
 \alpha_\infty(\varphi_0) = 0, \hone \nu_{\max}(\beta(\mathbf{N}_\infty)(\varphi_0)) \leq 0,
 \end{equation}
provided $\epsilon \ll C^{-1}$ in Proposition \ref{prop: iteration}.
\end{propx}

\subsection{Proof of Theorem \ref{thm: main_thm_parametrized} } 

Theorem \ref{thm: main_thm_parametrized} is now a direct consequence of the previous propositions.

\begin{proof}[Proof of Theorem \ref{thm: main_thm_parametrized}] Assuming Propositions \ref{prop: iteration}, \ref{prop: vanishing_counter_term} and taking $\varphi_0$ as in Proposition \ref{prop: vanishing_counter_term}, $H_0(\varphi_0, \cdot)$ admits an invariant $d$-dimensional torus with rotation vector $\omega$ parametrized by $\overline{\phi}: \T^d \rightarrow \T^{d + l} \times \R^{d + l}$ 
\[\overline{\phi}(q) = \Phi^\infty(\varphi_0, q, 0, 0, 0)\]
provided that $\| f_0 \|_{\T^l \times \initialmod}$ is sufficiently small. The bound on $\left\| \overline{\phi} - \phi_0\right\|_{\T^d}$ follows directly from the estimates in Proposition \ref{prop: iteration}.
\end{proof}

The remaining of this work concerns the proof of Propositions \ref{prop: iteration} and \ref{prop: vanishing_counter_term}.

\section{Proof of Proposition \ref{prop: vanishing_counter_term}}
\label{sc: proof_B}

In the following $D$ and $D_\varphi$ will denote, respectively, the differential operators $(\partial_q, \partial_x,$ $\partial_p, \partial_y)$ and $(\partial_{\varphi_1}, \cdots, \partial_{\varphi_l})$. Both $D$ and $D_\varphi$ define row vectors. To simplify the notation we denote by $\nabla$ and $\nabla_\varphi$ the associated column vectors, namely, $\nabla f = Df^T,$ $\nabla_\varphi f = D_\varphi f ^T$ for any $f \in C^{1, \omega}(\T^l \times D_{r, s})$.  Given a function $\Phi$ taking values in $D_{r, s}$ we denote its projections to the coordinates $q, x, p, y$ by $\Phi_q, \Phi_x, \Phi_p, \Phi_y$ respectively.  Inspired by the notation used in \cite{poschel_lecture_2001}, we write 
\[ u \precdot v\]
if there exists a positive constant $C$ depending only on $\deppp$ such that $u \leq Cv$. Similarly, given $F, G \in C^{\infty, \omega} (\T^l \times D_{r,s})$ we write
 \[ F = G + O(\epsilon),\]
if there exists a positive constant $C$ depending only on $\deppp$ such that
\[ \| F - G\|_{2, r, s} < C \epsilon.\]
To motivate the proof of Proposition \ref{prop: vanishing_counter_term}, let us say a few words about the relation between $\alpha_n$ and $\beta_n$ in Proposition \ref{prop: iteration}. It will be clear from the construction that in the first step of the iterative procedure
\[ \alpha_1(\varphi) = \Avg \partial_x f_0(\varphi, 0), \hone \beta_1(\varphi) = \Avg \partial^2_x f_0 (\varphi, 0). \]
Moreover, since $f_0$ satisfies (\ref{eq: equal_derivatives}), it readily follows that
\begin{gather}
\label{eq: alpha_1}
\alpha_1(\varphi) = \nabla_\varphi \Avg f_0(\varphi, 0), \\
\label{eq: beta_1}
\beta_1(\varphi) = \nabla_\varphi \alpha_1(\varphi) = \nabla_\varphi \Avg f_0(\varphi, 0).
\end{gather}
If we denote $\zeta_1(\varphi) = \Avg f_0(\varphi, 0)$, it follows that any local maximum $\varphi_1$ of $\zeta_1$ satisfies
\[ \alpha_1(\varphi_1) = 0, \hone \nu_{\max}(\beta_1(\varphi_1)) \leq 0.\]
Notice that if analogous of equations (\ref{eq: alpha_1}) and (\ref{eq: beta_1}) were verified for all $n \geq 1$, Proposition \ref{prop: vanishing_counter_term} would follow. Although this is not always the case, we will show that the counter-term $\alpha_n$ is actually $\epsilon_n$-close to 
the gradient of a function $\zeta_n$ whose Hessian is explicitly related, up to a term of order $\epsilon_n$, with $\beta_n$ (see equation (\ref{eq: beta})). Assuming Proposition \ref{prop: iteration} holds, and using the notations there introduced, let us denote 
\begin{gather*}
H_0 = N_0 + f_0 \hone F = H_0 - \langle \omega, p \rangle, \\
\mathbf{N}_n = (\omega, c_n, \beta_n, \Gamma_n, M_n, I_l, g_n, h_n),\\
\mathbf{N}_\infty = (\omega, c_\infty, \beta_\infty, \Gamma_\infty, M_\infty, I_l, g_\infty, h_\infty).
\end{gather*}
Let $\partial_\omega$  denote the differential operator given by
\[
\begin{array}{cccc}
\partial_\omega: &  C^1(\T^d) &  \rightarrow &  C(\T^d) \\
& h &  \mapsto & \langle \omega, \partial_q h \rangle
\end{array}.
\]
Define $\zeta_n, \zeta \in C^\infty(\T^l)$ as
\begin{gather*}
\zeta_n (\varphi) = \Avg \left( F \circ \Phi^n + \langle \Phi^n_p, \omega - \partial_\omega \Phi^n_q \rangle - \langle \Phi^n_y , \partial_\omega \Phi^n_x \rangle \right)(\varphi, 0), \\
\zeta(\varphi) = \Avg \left( F \circ \Phi^\infty + \langle \Phi^\infty_p, \omega - \partial_\omega \Phi^\infty_q \rangle - \langle \Phi^\infty_y , \partial_\omega \Phi^\infty_x \rangle \right)(\varphi, 0),
\end{gather*}
for all $n \geq 1$.  We claim that the functions defined in Proposition \ref{prop: iteration} satisfy the following.

\begin{prop}
\label{prop: alpha}
For all $n \geq 1$ and all $\varphi \in \T$ satisfying $\nu_{\max}(\beta_{n-1}(\varphi)) \leq \delta_{n-1}$ we have 
\begin{equation}
\label{eq: alpha}
\left| \alpha_n(\varphi) - \nabla_\varphi \zeta_n(\varphi) \right|, \left| D_\varphi \alpha_n(\varphi) - \nabla_\varphi ^2\zeta_n(\varphi) \right| \precdot \epsilon_{n}.
\end{equation}

\end{prop}
\begin{prop}
\label{prop: beta}
There exist sequences $L_n, R_n \in C^{\infty}(\T^l, M_{l \times l}(\R))$ such that for all $n \geq 1$ and all $\varphi \in \T$ satisfying $\nu_{\max}(\beta_{n-1}(\varphi)) \leq \delta_{n-1}$ 
\begin{equation}
\label{eq: bounds_LR}
|L_n(\varphi) - I_l|, |R_n^T(\varphi) - I_l| \precdot \epsilon_0
\end{equation}
and
\begin{equation}
\label{eq: beta}
\left| \beta_n (\varphi) - \Gamma_nM_n^{-1} \Gamma_n^T (\varphi) - L_n(\varphi) \nabla_\varphi \alpha_n (\varphi) R_n(\varphi) \right| \precdot \epsilon_{n}.
\end{equation}
Furthermore $L_n, R_n$ converge to well defined functions $L, R \in C^1(\T^l, M_{l \times l}(\R))$ as $n$ goes to infinity. 
\end{prop}

Before proving Propositions \ref{prop: alpha}, \ref{prop: beta} let us show how they imply Proposition \ref{prop: vanishing_counter_term}. 

\begin{proof}[Proof of Proposition \ref{prop: vanishing_counter_term}]. In the following we will use freely the notations in the statement of Proposition \ref{prop: iteration}. 
Let $C$ be the maximum of the constants given by Propositions \ref{prop: iteration}, \ref{prop: alpha} and \ref{prop: beta}. By Proposition \ref{prop: iteration} 
\[ \| \Phi^\infty - \id\|_{2, r_0/2, s_0/2} \precdot \epsilon^{1/2}.\]
Let $\varphi_0$ be one of the points where $\zeta$ attains its maximum. By definition and the bounds in Proposition \ref{prop: iteration}
\[ \| \zeta - \zeta_n\|_{C^2} \precdot \epsilon_{n}^{\frac{1}{2}},\]
for all $n \in \N$. Hence
\begin{gather}
\label{eq: zeta_gradient}
\nabla \zeta(\varphi_0) = 0, \hone | \nabla_\varphi \zeta_n(\varphi_0) | \precdot \epsilon_{n}^{\frac{1}{2}}, \\
\label{eq: zeta_hessian}
\nu_{\max}(\nabla^2 \zeta(\varphi_0)) \leq 0, \hone \nu_{\max}(\nabla_\varphi^2\zeta_n(\varphi_0)) \precdot \epsilon_{n}^{\frac{1}{2}}.
\end{gather}

\textbf{Claim.} Let $n \geq 1$. If $\nu_{\max}(\beta_{n - 1}(\varphi_0)) \leq \delta_{n - 1}$ then $\nu_{\max}(\beta_n(\varphi_0)) \precdot \epsilon_n^{1/2}$. 
\begin{proof}[Proof of the Claim.] In the following all the functions are evaluated at $\varphi_0$. We omit the evaluation point to simplify the notation. Since $\beta - \Gamma_nM^{-1}\Gamma_n$ is a symmetric matrix $\nu_{\max}$ is equal to its greatest eigenvalue. By Propositions \ref{prop: alpha} and \ref{prop: beta}
\begin{equation}
\label{eq: bound_nu}
\nu_{\max} (\beta_n - \Gamma_n M^{-1} \Gamma_n) - \max \{ \text{Re}(\lambda) \mid \lambda \in \sigma(L_n  \nabla_\varphi^2 \zeta_nR_n)\} \precdot \epsilon_n.
\end{equation}
Let $\lambda \in \sigma(L_n\nabla^2_\varphi \zeta_n R_n)$ and let $v \in \mathbb{C}^{d} \,\setminus \,\{ 0 \}$ be an eigenvector associated to it. Then 
\[ R_n^T \nabla^2_\varphi \zeta_n R_n v = \lambda R_n^T (L_n)^{-1} v,\]
which implies
\[ \lambda = \frac{\langle \nabla^2_\varphi \zeta_n R_n v, R_n \overline{v} \rangle}{\langle (L_n)^{-1} v, R_n \overline{v} \rangle}.\]
Notice that the numerator in the last equation is always real. By (\ref{eq: bounds_LR})
\[|\langle (L_n)^{-1}v, R_n\overline{v} \rangle - 1 | \precdot \epsilon_0.\]
Thus
\[\text{Re}(\lambda) = \frac{\langle \nabla_\varphi \zeta_n R_n v, R_n \overline{v} \rangle \text{Re} \langle (L_n)^{-1}v, R_n\overline{v}\rangle}{ |\langle (L_n)^{-1}v, R_n \overline{v} \rangle|^2} \precdot \nu_{\max}(\nabla^2_\varphi \zeta_n),\]
which yields to
\[ \max \{ \text{Re}(\lambda) \mid \lambda \in \sigma(L_n \nabla_\varphi \zeta_nR_n)\} \precdot \nu_{\max}(\nabla^2_\varphi \zeta_n).\]
By (\ref{eq: bound_nu})
\[ \nu_{\max} (\beta_n - \Gamma_n M^{-1} \Gamma_n) - \nu_{max}(\nabla^2_\varphi \zeta_n) \precdot \epsilon_n.\]
Since $ \Gamma_n M^{-1} \Gamma_n$ is negative definite and by (\ref{eq: zeta_hessian}) the claim follows.
\end{proof}

By definition $\delta_n > \epsilon_n^{1 / 3}$. Since $\nu_{\max}(\beta_0(\varphi_0) )= 0$ it follows from the previous claim that $\nu_{\max}(\beta_n(\varphi_0)) \leq \delta_n$
for all $n \in \N$. By Proposition \ref{prop: alpha} and (\ref{eq: zeta_gradient})
$|\alpha_n(\varphi_0)| \precdot \epsilon_n^{1/2}$
for all $n \in \N$. By making $n$ tend to infinity it follows that
\[ \alpha_\infty(\varphi_0) = 0, \hone \nu_{\max}(\beta_\infty(\varphi_0)) \leq 0.\]
\end{proof}

We now prove Propositions \ref{prop: alpha} and \ref{prop: beta}. 

\begin{proof}[Proof of Proposition \ref{prop: alpha}.] By (\ref{step_conjugacy}) and the symplectic nature of $\Phi^n$
\[ X_{H_0 - \alpha_n \cdot x} \circ \Phi^n = D\Phi^n X_{N_n + f_n}, \]
or equivalently
\begin{equation}
\label{symplectic_step}
 J(-\Omega_n + \nabla F \circ \Phi^n) = D\Phi^n J\nabla(N_n + f_n),\end{equation}
where
\begin{equation}
\label{Omega}
 J = \left( \begin{array}{cc} 0 & I_m \\ -I_m & 0 \end{array} \right), \hone \Omega_n(\varphi) = \left( \begin{array}{cccc} 0 \\ \alpha_n(\varphi) \\ -\omega \\ 0 \end{array} \right).
 \end{equation}
Let $\varphi \in \T^l$ satisfying $\nu_{\max}(\beta_{n-1}(\varphi)) \leq \delta_{n-1}.$ In the following and unless otherwise explicitly stated all the functions are evaluated at $(\varphi, q, 0)$. To simplify the notation we will not write the evaluation point. By (\ref{symplectic_step}) 
\[ \nabla F \circ \Phi^n = \Omega_n - J\partial_\omega \Phi^n + O(\epsilon_n).\]
Let
 \begin{equation}
 \label{eq: W_def}
 W_n(\varphi, q) = \left( \begin{array}{cccc} D_\varphi\Phi^n_q \\ I_l + D_\varphi\Phi^n_x \\ D_\varphi\Phi^n_p \\ D_\varphi\Phi^n_y \end{array} \right) \in M_{2m \times l}(\R).
 \end{equation}
By hypothesis $\Avg(\partial_x F) = \Avg(\partial_\varphi F)$ which yields to
\[ D_\varphi \Avg(F \circ \Phi^n) = \Avg((DF \circ \Phi^n)W_n).\]
By Proposition \ref{prop: iteration}, $|\Avg\Phi^n_x(\varphi, 0)| < \epsilon_n$. Hence
\begin{align*}
\nabla_\varphi\Avg(F \circ \Phi^n) & = \Avg(W_n^T (\nabla F \circ \Phi^n)) \\
& = \Avg(\alpha_n - \nabla_\varphi \Phi^n_p \omega - \nabla_\varphi \Phi^n J \partial_\omega \Phi^n \\
& \quad + \nabla_\varphi \Phi^n_x \alpha_n - J \partial_\omega \Phi^n_x ) + O(\epsilon_n) \\
& = \alpha_n - \nabla_\varphi\Avg(\langle \omega, \Phi^n_p \rangle) + \Avg(\nabla_\varphi \Phi^n J \partial_\omega \Phi^n ) \\
& \quad + O(\epsilon_n).
\end{align*}
%
%
%
%
Thus
\begin{equation}
\label{eq: gradient_form}
\alpha_n = \nabla_\varphi\Avg( F \circ \Phi^n +\langle \omega, \Phi^n_p \rangle) - \Avg(\nabla_\varphi \Phi^n J \partial_\omega \Phi^n) + O(\epsilon_n).
\end{equation}
Integrating by parts it follows that
\[ \Avg(h_1 \partial_\omega h_2) = - \Avg(h_2 \partial_\omega h_1),\]
for all $h_1, h_2 \in C^1(\td)$. Therefore
\begin{align*}
\Avg( \nabla_\varphi \Phi^n J \partial_\omega \Phi^n) & = \Avg( \nabla_\varphi \Phi^n_q \partial_\omega \Phi^n_p - \nabla_\varphi \Phi^n_p \partial_\omega\Phi^n_q \\
& \quad + \nabla_\varphi \Phi^n_x \partial_\omega \Phi^n_y - \nabla_\varphi \Phi^n_y \partial_\omega \Phi^n_x ) \\
& = \Avg( (\nabla_\varphi \partial_\omega \Phi^n_q)\Phi^n_p + \nabla_\varphi \Phi^n_p \partial_\omega\Phi^n_q\\
& \quad + (\nabla_\varphi \partial_\omega \Phi^n_x) \Phi^n_y + \nabla_\varphi \Phi^n_y \partial_\omega \Phi^n_x ) \\
& = \nabla_\varphi \Avg( \langle \Phi^n_p, \partial_\omega \Phi^n_q \rangle + \langle \Phi^n_y , \partial_\omega \Phi^n_x \rangle ).
\end{align*}
Replacing this in (\ref{eq: gradient_form}) we obtain 
\begin{align*}
 \alpha_n & = \nabla_\varphi \Avg \left( F \circ \Phi^n + \langle \Phi^n_p, \omega - \partial_\omega \Phi^n_q \rangle - \langle \Phi^n_y , \partial_\omega \Phi^n_x \rangle \right) + O(\epsilon_n) \\
 & =  \nabla_\varphi\zeta_n + O(\epsilon_n).
 \end{align*}
\end{proof}

\begin{proof}[Proof of Proposition \ref{prop: beta}.] Applying $D$ and $D_\varphi$ to equation (\ref{symplectic_step}) and evaluating at a point $(\varphi, q, 0)$, for some $\varphi \in \T$ satisfying $\nu(\beta_{n-1}(\varphi)) \leq \delta_{n-1}$, we obtain
\begin{gather}
\label{Z_equation}
\partial_\omega Z_n = J D^2 F \circ \Phi^n Z_n - Z_n D^2 N_n + O(\epsilon_n), \\
\label{W_equation}
\partial_\omega W_n = J D_\varphi \Omega_n + J D^2 F \circ \Phi^n W_n + O(\epsilon_n),
\end{gather}
where $Z_n(\varphi, q) = D \Phi^n(\varphi, q, 0) \in M_{2m}(\R)$ and $\Omega_n, W_n$ as in (\ref{Omega}), (\ref{eq: W_def}) respectively. Like in the proof of Proposition \ref{prop: alpha} we have omitted, and will omit in the following, the evaluation point $(\varphi, q, 0)$ in all the expressions. Since $\Phi^n$ is symplectic 
 \begin{equation}
 \label{symplectic}
 J Z_n^T = Z_n^{-1}J.
 \end{equation} 
 Thus equation (\ref{Z_equation}) implies
\begin{equation}
\label{Z_inverse_equation}
\partial_\omega Z^{-1}_n = J D^2N_n Z^{-1}_n(q) - Z^{-1}_n J D^2 F \circ \Phi + O(\epsilon_n).
\end{equation}
From (\ref{W_equation}) and (\ref{Z_inverse_equation})
\begin{equation}
\label{ZW_equation}
\partial_\omega (Z^{-1}_n W_n) = J D^2 N_n Z^{-1}_n W_n + Z^{-1}_n J D_\varphi \Omega_n + O(\epsilon_n).
\end{equation}
Averaging over $\td$ in (\ref{ZW_equation}) and by (\ref{symplectic})
\[ \Avg(Z_n^T) D_\varphi \Omega_n = - D^2 N J \Avg(Z_n^T J W_n) + O(\epsilon_n),\]
which written in matrix form yields to
\begin{align*}
\Avg \left( \begin{array}{cccc} 0 \\ (\partial_x \Phi_x^n)^T D_\varphi \alpha_n \\ (\partial_p \Phi_x^n)^T D_\varphi \alpha_n \\ (\partial_y \Phi_x^n)^T D_\varphi \alpha_n \end{array} \right) & = \Avg \left( \begin{array}{cccc} 
0 & 0 & 0 & 0 \\ 
-\Gamma_n & 0 & 0 & \beta_n \\
-M_n & 0 & 0 & \Gamma_n^T \\
0 & -I_l & 0 & 0 \\
\end{array} \right) \left( \begin{array}{cccc} \partial_q \Phi^n J W_n \\ \partial_x \Phi^n J W_n \\ \partial_p \Phi^n J W_n \\ \partial_y \Phi^n J W_n \end{array} \right)  + O(\epsilon_n).
\end{align*}
From last equation it follows that 
\begin{align*} 
(\beta_n - \Gamma_n M_n^{-1} \Gamma_n^T) \partial_y \Phi^n J W_n & = (\Avg(\partial_x \Phi^n_x)^T - \Gamma_n M_n^{-1} \partial_p \Avg \Phi^n_x) D_\varphi \alpha_n  + O(\epsilon_n).
\end{align*}
Notice that
\[\partial_y \Phi^n J W_n = I_l + O(\epsilon_0), \hone (\partial_x \Phi^n_x)^T = I_l + O(\epsilon_0), \hone \partial_p \Avg \Phi^n_x = O(\epsilon_0).\]
Thus $\Avg(\partial_y \Phi^n J W_n)$ is invertible provided $\epsilon_0$ is sufficiently small. Let us denote 
\[R_n = \Avg(\partial_y \Phi^n J W_n)^{-1}, \hone L_n = \Avg(\partial_x \Phi^n_x)^T - \Gamma_n M_n^{-1} \partial_p \Avg \Phi^n_x.\]
Then $|R_n - I_l|, |L_n - I_l| \precdot \epsilon_0$ and we have
\[ \beta_n - \Gamma_n M_n^{-1} \Gamma_n^T = L_n \nabla_\varphi \alpha_n R_n + O(\epsilon_n). \]
Notice that the definitions of $L_n$ and $R_n$ make sense even if $\varphi \in \T$ does not satisfy $\nu(\beta_{n-1}(\varphi)) \leq \delta_{n-1}$. Thus we obtain well defined functions $L_n, R_n \in C^1(\T^l, M_l(R))$ for all $n \in \N$, satisfying the conclusions of Proposition \ref{prop: alpha}. By the bounds in Proposition \ref{prop: iteration}, the sequences $L_n$ and $R_n$ converge to well defined functions $L, R \in C^1(\T^l, M_l(\R))$.
\end{proof}

\section{Proof of Proposition \ref{prop: iteration}}
\label{sc: proof_A}
The rest of the paper will concern the proof of Proposition \ref{prop: iteration}. Technical lemmas concerning small divisors problems and bounds for composite functions are proven in Section \ref{sc: lemmata}.

 Let us recall that given a smooth Hamiltonian $H$ we denote its associated Hamiltonian flow by $\Psi_H^t$.  As an abuse of notation, given $v \in \R^m$ we denote by $X_{H + v \cdot (q, x)}^t$ the vector field $X_{H + v \cdot (q, x)} = (\partial_{p, y} H, -\partial_{q, x}H - w)$ and by $\Psi_{H + v \cdot (q, x)}^t$ the associated flow. Recall that for any $t_0 \in \R$ fixed, the transformations $\Psi_H^{t_0}$ and $\Psi_{H + v \cdot (q, x)}^{t_0}$ are well-defined symplectomorphisms. A formal definition of these objects is given in the Appendix.

 The following proposition amounts to one step of the iterative scheme described in Proposition \ref{prop: iteration}. 

\begin{prop}
\label{KAM_Step}
There exist $\kappa(d, l, \tau) > 0$ and $\lambda(\deppp) > 0$, such that for any $ 0 < \sigma \leq \frac{1}{10}\min\{ r, s \}$, any  $\epsilon,$ $\delta > 0$ satisfying 
 \begin{gather}
\label{eq: epsilon_cond}
 \epsilon^{1/2} |\log \epsilon|^{4(l + 2)\tau} < \lambda \sigma^{\kappa}, 
\\
\label{eq: delta_cond}
\frac{1}{8}\left( \frac{\sigma}{4|\log \epsilon|}\right)^{2\tau} < \delta,
\end{gather} 
and any $\phi \in C^{\infty, \omega}(\T^l \times D_{r,s}, \C^l),$ $\mathbf{N} \in \mathcal{N}^{\omega, \delta}_{r,s}$, $f \in C^{\infty,\omega}(\T^l \times D_{r,s})$
obeying
\begin{gather*}
Q(\mathbf{N}) = I_l, \\
\max \Big\{  \| \phi - x \|_{\rsnorm}, \|\mathbf{N} - \mathbf{N}_0\|_{\mathcal{N}_{r, s}} \Big\} < 2\lambda,   \\
\max \Big\{  \| f\|_{\rsnorm}, \|\Avg \phi(\varphi, 0)\|_{C^2(\T^l)} \Big\}  < \epsilon, 
\end{gather*}
there exist 
\[ \mathbf{N}^+ \in \mathcal{N}^{\omega, \delta^+}_{r^+, s^+}, \hskip0.5cm \alpha \in C^\infty(\T^l, \R^l), \hskip0.5cm f^+ \in C^{\infty, \omega}(\T^l \times D_{r^+, s^+}),\]
\[ v \in C^\infty(\T^l, \R^d), \hskip0.5cm F \in C^{\infty, \omega}(\T^l \times D_{r - 8\sigma, s}), \hskip0.5cm \Psi \in C^{\infty, \omega}(\T^l \times D_{r ^+, s^+}, \dom),\]
with $r^+ = r - 10\sigma,$ $ s^+ = s - \sigma$, $Q(\mathbf{N}^+) = 0$ and $\Psi = \Psi_{F + v \cdot q}^1$, obeying
\begin{gather*}
\max\Big\{   \| \alpha \|_{C^2}, \|v \|_{C^2},  \| \mathbf{N}^+ - \mathbf{N}\|_{\mathcal{N}_{r^+, s^+}},  \| \Psi - \id\|_{\rspnorm},  \| F \|_{C^{2, 3}(\T^l \times D_{r - 8\sigma, s})}   \Big\} < \epsilon^{\frac{1}{2}}, \\
\max \Big\{ \| f^+\|_{\rspnorm},  \|\Avg (\phi \circ \Psi) (\varphi, 0)\|_{C^2(\T^l)} \Big\} < \epsilon^{\frac{3}{2}},
\end{gather*}
 such that 
\[
(N + f - \langle \alpha, \phi \rangle )\circ \Psi = N^+ + f^+ .\]
\end{prop} 

\begin{proof}
Let $0 < \lambda < 1$ and $\kappa > 0$. We will show that Proposition \ref{KAM_Step} holds provided that $\lambda$ is sufficiently small and $\kappa$ sufficiently large. As we shall see the smallness conditions on $\lambda$, which will appear naturally along the proof, depend only on $\deppp$. The value of $\kappa$ will depend only on the constant $\kappa_0$ given by Lemma \ref{lem: cohomological_equation}. 

By Lemma \ref{trig_approx}, we can suppose WLOG that the functions $\phi, f$ are trigonometric polynomials in the $q$ variable of degree at most
\[K = \frac{4|\log \epsilon|}{\sigma}.\]
This approximation will allow us to solve certain small divisors problems which will appear in the following. As stated in the proposition, the symplectic transformation $\Psi$ will be the time one map of the Hamiltonian flow associated to a generating function $F$. This function is obtained as the solution of an appropriate \textit{cohomological equation} (\ref{eq: cohomological_equation}) given by the following lemma, whose proof we postpone to the end of this section.
\begin{lem}
\label{lem: cohomological_equation}
There exist positive constants $\kappa_0(d, l, \tau),$ $C(\deppp)$ and functions $\alpha \in C^\infty(\T^l, \R^l),$ $v \in C^\infty(\T^l, \R^d),$ $F \in C^{\infty, \omega}(\T^l \times D_{r - 8\sigma,s})$, $\overline{\mathbf{N}} \in \mathcal{N}^{0, 2\delta_+}_{r - 8\sigma, s}$ obeying
\begin{gather*}
Q(\overline{\mathbf{N}}) = 0, \hskip1cm  \Avg(\phi + \{ \phi, F + v \cdot q\}(\varphi, \cdot , 0)) = 0,\\
\max\Big\{ \|\alpha\|_{C^2}, \| v\|_{C^2}, \| \overline{\mathbf{N}}\|_{\mathcal{N}_{r - 8\sigma, s}} , \| F\|_{C^{2, 3}(\T^l \times D_{r - 8\sigma, s})} \Big\} < \frac{C\epsilon}{\delta_+^{l + 2}\sigma^{\kappa_1}},
\end{gather*}
such that
\begin{equation}
\label{eq: cohomological_equation}
f - \langle \alpha, \phi \rangle + \{ N - g(\mathbf{N}) ,F + v \cdot q\} = \overline{N},
\end{equation}
where $\{\cdot, \cdot\}$ denotes the Poisson bracket in $\T^d \times \R^l \times \R^d \times \R^l.$
\end{lem}
Let $F, v, \alpha, \overline{\mathbf{N}}, \kappa_0$ as in Lemma \ref{lem: cohomological_equation} and define
\[ \kappa(d, l, \tau) = 4\kappa_0.\] 
Let $\Psi = \Psi^1_{F + v \cdot q}$. By (\ref{eq: epsilon_cond}), the bounds of Lemma \ref{lem: cohomological_equation} and Lemma \ref{composition_bounds}, the transformation $\Psi: \T^l \times D_{r^+, s^+} \rightarrow \dom$ is well defined and satisfies
\[ \| \Psi - \id \|_{2, r^+, s^+} \leq \epsilon^{\frac{1}{2}},\]
for $\lambda$ sufficiently small. By (\ref{eq: cohomological_equation}) and (\ref{eq: composition_formula})
 \begin{align*}
(N + f & - \langle \alpha, \phi \rangle ) \circ \Psi = N + \overline{N} + g(\mathbf{N}) \circ \Psi - g(\mathbf{N}) \\
& \quad + \int_0^1 \{ (1- t)\overline{\mathbf{N}} + t(f - \langle \alpha, \phi \rangle ), F + v \cdot q\} \circ \Psi^t_{F + v \cdot q} dt. 
\end{align*}
Denote
\[ \mathbf{N}^+ = \mathbf{N} + \overline{\mathbf{N}} + (g(\mathbf{N}) \circ \Psi - g(\mathbf{N}))e_g,\]
\[f^+ = \int_0^1 \{ (1- t)\overline{\mathbf{N}} + t(f - \langle \alpha, \phi \rangle ), F + v \cdot q\} \circ \Psi^t_{F + v \cdot q} dt ,\]
where $e_g$ denotes the vector in $\normalrsp$ with 1 in the $g$-coordinate and zero everywhere else. Notice that $\mathbf{N}^+ \in \normalrsp$ and $f^+ \in C^{\infty, \omega}(\T^l \times D_{r^+, s^+})$. Moreover, 
\[ w(\mathbf{N}^+) = \omega, \hskip1cm Q(\mathbf{N}^+) = 0, \]
\[ (N + f - \langle \alpha, \phi \rangle)\circ \Psi = N^+ + f^+ .\]
Let us check that the bounds in the statement hold. It follows from Lemma \ref{lem: cohomological_equation}, (\ref{eq: epsilon_cond}) and Cauchy's estimates that 
\[
 \| \alpha\|_{C^{2}(\T^l)} \leq \epsilon^{\frac{1}{2}}, \hone \| f^+\|_{\rspnorm} \precdot \dfrac{\epsilon^2}{\delta_+^{2(l + 2)}\sigma^{2\kappa_1 + 2}} \leq \epsilon^{\frac{3}{2}},\]
for $\lambda$ sufficiently small. Since
\begin{align*}
 \| g(\mathbf{N}) \circ \Psi - g(\mathbf{N}) \|_{\rspnorm} & \precdot \left\| \int_0^1 \{ g, F + v\cdot q\} \circ \Psi^t_{F + v \cdot q} dt \right\|_{\rspnorm} \\
 & \precdot \dfrac{\epsilon}{\delta_+ \sigma^{\kappa_1 + 2} },
 \end{align*} 
 it follows that 
 \[
 \| \mathbf{N}^+ - \mathbf{N}\|_{\mathcal{N}_{r^+, s^+}} \leq \epsilon^{\frac{1}{2}}.
 \] 
In particular $\nu_{\max}(\beta(\mathbf{N}^+)(\varphi)) < \delta_+$ implies 
\[\nu_{\max}(\beta(\mathbf{N})(\varphi)) < 2\delta_+ < \delta.\]
Thus if $\nu_{\max}(\beta(\mathbf{N}^+)(\varphi)) < \delta_+$ then $g(\mathbf{N}^+)$ vanishes together with its derivatives of all orders. Hence $\mathbf{N}^+ \in \normalrsp^{\omega, \delta_+}.$ It remains to check that 
 \[ \|\Avg(\phi \circ \Psi)(\varphi, 0)\|_{C^2(\T^l)} \leq \epsilon^{\frac{3}{2}}.\] 
 By (\ref{eq: composition_formula}) and Lemma \ref{lem: cohomological_equation}
  \begin{align*}
\Avg(\phi \circ \Psi)(\varphi, 0) & = \Avg\Bigg( \phi + \{ \phi , F + v \cdot q\} + \int_0^1  \{ \{ \phi , F + v \cdot q\}, F + v\cdot q \} \circ \Psi^t_F dt  \Bigg)(\varphi, 0) \\
& =  \Avg\Bigg( \int_0^1  \{ \{ \phi , F + v \cdot q\}, F + v\cdot q \} \circ \Psi^t_F dt  \Bigg)(\varphi, 0).
 \end{align*}
Hence, by Lemma \ref{lem: cohomological_equation}
 \begin{align*}
 \|\Avg(\phi \circ \Psi)(\varphi, 0)\|_{C^2(\T^l)}  & \precdot \frac{\epsilon^2}{\delta_+^{2(l + 2)} \sigma^{2\kappa_1}} \\
 & \leq \epsilon^{\frac{3}{2}},
 \end{align*}
 for $\lambda$ sufficiently small. This completes the proof. 
 \end{proof}

We now prove Proposition \ref{prop: iteration} by iterating Proposition \ref{KAM_Step}. 

\begin{proof}[Proof of Proposition \ref{prop: iteration}] Let $\kappa, \lambda$ as in Proposition \ref{KAM_Step}. Define
\[ \sigma_0 = \frac{1}{10} \min \{ r, s\}, \hone \epsilon_0 = (\lambda \sigma_0)^{3\kappa}. \]
Taking $\lambda$ smaller if necessary we can suppose WLOG that the function $$x \mapsto \epsilon_0^{1/2} |\log (\epsilon_0)|^{4(l + 2)\tau}$$ is increasing in the interval $(0, \epsilon_0]$. Similarly, we can assume that
\[ 
\epsilon_0^{1/2} |\log (\epsilon_0)|^{4(l + 2)\tau} < \epsilon_0^{1/3} = \lambda \sigma_0^{3 \kappa}
\]
holds. Let $r_0 = r,$ $s_0 = s,$ and define for all $n \in \N$
 \[ \sigma_n = \frac{1}{10^{n + 1}}\min\{ r, s\}, \hone \delta_n = \left( \frac{\sigma_n}{|\log \epsilon_n|}\right)^{4\tau},\]
 \[ r_{n + 1} = r_{n} - 10\sigma_n, \hone s_{n + 1} = s_{n} - \sigma_n.
\]
We will construct the sequences $\mathbf{N}_n, \alpha_n, f_n, \Phi^n$ recursively as follows. Assume that for some $n \in \N$ the functions $\mathbf{N}_n, \alpha_n, f_n, \Phi^n$ have been defined and that Proposition \ref{KAM_Step} can be applied to
\begin{equation}
\label{eq: n_hypotheses}
\begin{aligned}
\epsilon = \epsilon_n, \hone \delta = \delta_n, \hone \sigma = \sigma_n, \\
 \phi = \Phi^n_x, \hone \mathbf{N} = \mathbf{N}_n, \hone f = f_n. 
 \end{aligned}
\end{equation}
Notice that for $n = 0$, that is, for $N_0,$ $\alpha_0,$ $f_0,$ and $\Phi^0$ as in the statement of Proposition \ref{prop: iteration}, the hypotheses of Proposition \ref{KAM_Step} are trivially verified. 
Denote by $\mathbf{N}^{(n + 1)},$ $\alpha^{(n + 1)},$ $f^{(n + 1)},$ $\Phi^{(n + 1)}$ the functions given by the Proposition \ref{KAM_Step} when applied to (\ref{eq: n_hypotheses}) and by $F^{(n + 1)}$ the generating function of $\Phi^{(n + 1)}$. We define
\[ \mathbf{N}_{n + 1} = \mathbf{N}^{(n + 1)}, \hone \alpha_{n+1} = \alpha_n + \alpha^{(n+1)},\]
 \[ f_{n + 1} = f^{(n + 1)}, \hone \Phi^{n+1} = \Phi^n \circ \Phi^{(n+1)}.\]
Thus, if we can iterate this process $n$ times, by Proposition \ref{KAM_Step} we have $\mathbf{N}_{n} \in \mathcal{N}_{r_{n}, s_{n}}^{\omega, \delta_{n}}$, $Q(\mathbf{N}_{n} ) = I_l$ and  
\begin{equation}
\label{n_step}
\begin{aligned}
(N_0 + f_0 - \langle \alpha_{n}, x \rangle) \circ \Phi^{n} = N_n + f_n.
 \end{aligned}
 \end{equation}
Therefore, to prove the proposition, it suffices to show that this recursive process can be iterated indefinitely and that the bounds in the statement are verified.  

\textit{Recursion.} 
We will show that if $\epsilon_{n - 1}, \delta_{n - 1}, \sigma_{n - 1}, \Phi^{n - 1}_x, \mathbf{N}_{n - 1}, f_{n - 1}$ satisfy the hypotheses of Proposition $\ref{KAM_Step}$ the same is true for $\epsilon_{n}$, $\delta_{n}$, $\sigma_{n},$ $\Phi^{n}_x,$ $\mathbf{N}_{n},$ $f_{n}.$ By construction 
\[\| f_n\|_{\rsnnorm}, \| \Avg \Phi^{n}_x(\varphi, 0) \|_{C^2} < \epsilon_{n}.\] By the bounds in Proposition \ref{KAM_Step}
\begin{align*}
 \| \Phi^{n}_x - x \|_{\rsnnorm}, \| \mathbf{N}_{n} - \mathbf{N}_0 \|_{\normalrsn} & < \sum_{k = 0}^{n - 1} \epsilon_k^{\frac{1}{2}} \\
 & < 2\epsilon_0^{\frac{1}{2}} \\
 & < 2\lambda. 
 \end{align*} 
From (\ref{eq: epsilon_cond})
\begin{align*}
\epsilon_{n}^{1/2}|\log \epsilon_{n}|^{4(l + 2)\tau} & < 2^{4(l + 2)\tau} \epsilon_{n - 1}^{3/2}|\log \epsilon_{n}|^{4(l + 2)\tau} \\
& \leq 2^{4\tau} \epsilon_{n - 1} \lambda \sigma_{n - 1}^{\kappa} \\
& \leq (2^{4\tau} 10^\kappa \lambda) \lambda \sigma_{n}^\kappa \\
& \leq \lambda \sigma_{n}^{\kappa},
\end{align*}
for $\lambda$ sufficiently small. Finally, equation (\ref{eq: delta_cond}) is equivalent to $\delta_{n} \leq 8\delta_{n - 1}$, which follows trivially from the definitions of $\delta_n, \epsilon_n$ for $\lambda$ sufficiently small. Thus $\epsilon_{n}$, $\delta_{n}$, $\sigma_{n},$ $\Phi^{n}_x,$ $\mathbf{N}_{n},$ $f_{n}$ satisfy the hypotheses of Proposition \ref{KAM_Step}. 

Therefore the functions $\mathbf{N}_n, \alpha_n, f_n, \Phi^n$ are well defined for all $n \in \N$. It remains to check that the bounds in Proposition \ref{prop: iteration} are satisfied. 

\textit{Bounds.} Fix $n \in \N$. By construction and Proposition \ref{KAM_Step}
\[ \| \mathbf{N}_{n + 1} - \mathbf{N}_{n}\|_{\normalrsnp} , \| \alpha_{n + 1} - \alpha_{n}\|_{C^2(\T^l)} < \epsilon^{\frac{1}{2}}_n, \]
\[ \| f_n\|_{\rsnnorm}, \| \Avg(\Phi^{n}_x(\cdot, 0)) \|_{C^2(\T^l)} < \epsilon_{n}, \]
for all $n \geq 1$. By (\ref{eq: composition_formula}) we have
\begin{align*}
 \| \Phi^{n + 1} - \Phi^{n}\|_{\rsnpnorm} &= \| \Phi^{n} \circ \Phi^{(n + 1)} - \Phi^{n}\|_{\rsnpnorm} \\
 & = \left\| \int_0^1 \left\{ \Phi^{n}, F^{(n + 1)} \right\} \circ \Psi_{ F^{(n + 1)}}^t dt \right\|_{\rsnpnorm} \\
 & \precdot \| \Phi^{n}\|_{\rsnnorm} \left\| F^{(n + 1)} \right\|_{C^{2, 3}(\T^l \times D_{r_n, s_n})} \\
 & \precdot \| \Phi^{n}\|_{\rsnnorm}\epsilon_n^{\frac{1}{2}}.
\end{align*} 
By construction, for all $k \in \N$
 \[\| \Phi^{(k)} - \id \|_{\rsnnorm} \leq \epsilon_k^{1/2}. \]
Thus
\[
\| \Phi^{n}\|_{\rsnnorm} = \| \Phi^{(n)} \circ \Phi^{(n - 1)} \circ \dots \circ \Phi^{(1)} \|_{\rsnnorm} \leq \prod_{k = 0}^{n - 1} \left(1 + \epsilon_k^{\frac{1}{2}} \right),
\]
which is uniformly bounded since the sequence $\epsilon_k$ is rapidly decreasing. Therefore 
\[ \| \Phi^{n + 1} - \Phi^{n}\|_{\rsnpnorm} \precdot \epsilon_n^{\frac{1}{2}}. \]
\end{proof} 

It remains to prove Lemma \ref{lem: cohomological_equation}.

\begin{proof}[Proof of Lemma \ref{lem: cohomological_equation}] Let us express the linear part (on the variables $x, p , y$) of $f, \phi$ as 
\begin{equation*}
f = a^f(\varphi, q) + b^f(\varphi, q)^T \cdot \left( \begin{array}{ccc} x \\ p \\ y \\ \end{array} \right) + O^2(x,p,y),
\end{equation*}
\[
\phi = a^\phi(\varphi, q) + b^\phi(\varphi, q)^T \cdot \left( \begin{array}{ccc} x \\ p \\ y \\ \end{array} \right) + O^2(x,p,y),
\]
Recall that by assumption the functions $f,$ and $\phi$ are trigonometric polynomials on the variable $q$ of degree at most
\[ K = \frac{4 |\log \epsilon|}{\delta}\]
in each coordinate. For $\alpha \in C^\infty(\T^l, \R^l)$ fixed, we write the quadratic part of $f -\langle \alpha, \phi \rangle$ as 
\begin{align*}
f - \langle \alpha, \phi \rangle & = a(\varphi, q) + b(\varphi, q)^T \cdot \left( \begin{array}{ccc} x \\ p \\ y \\ \end{array} \right) + \dfrac{1}{2}\left\langle d(\varphi, q) \left(\begin{array}{ccc} x \\ p \\ y \\ \end{array} \right) , \left( \begin{array}{ccc} x \\ p \\ y \\ \end{array} \right) \right\rangle \\
& \quad + O^3(x,p,y).
\end{align*}
Clearly 
\[ a = a^f - \langle \alpha, a^\phi \rangle, \hone b = b^f - \langle \alpha, b^\phi \rangle. \]
Consider $F$ of the form
\begin{equation}
\label{Generating}
F = A(\varphi, q) + B^T(\varphi, q) \cdot \left( \begin{array}{ccc} x \\ p \\ y \\ \end{array} \right) + \dfrac{1}{2}\left\langle D(\varphi, q) \left(\begin{array}{ccc} x \\ p \\ y \\ \end{array} \right) , \left( \begin{array}{ccc} x \\ p \\ y \\ \end{array} \right) \right\rangle ,
\end{equation}
for some $C^{\infty, \omega}$ functions $A, B, D$. These functions will be trigonometric polynomials on the variable $q$ of degree at most $K$ in each coordinate. Denote
\[ B = (B_x, B_p, B_y), \hone D = \begin{pmatrix}
D_{xx}& D_{xp} & D_{xy} \\
D_{px}& D_{pp} & D_{py} \\
D_{yx}& D_{yp} & D_{yy} \\
\end{pmatrix}.\]
To simplify the notation in the following we write
\[\mathbf{N} = (\omega, c, \beta, \Gamma, M, I_l, g, h).\]
Recall that the associated Hamiltonian is given by 
\begin{align*}
N(\varphi, q, x, p, y) &= c(\varphi) + \langle \omega, p \rangle + \dfrac{1}{2}\langle M(\varphi)p , p \rangle  + \dfrac{1}{2}|y|^2+ \langle \Gamma(\varphi) p, x \rangle + \dfrac{1}{2}\langle \beta(\varphi) x , x \rangle  \\
& + g(\varphi, q, x, p, y) + h(\varphi, q, x, p, y). 
\end{align*}
Let $P$ be the truncation of $\{ h, F + v \cdot q\}$ to terms of order at most $2$ in the variables $x, p, y$. Notice that $P$ only depends on $v, A, B_x, B_y$. Indeed
\begin{align*}
\{ h, F + v \cdot q\} & = \{ h, v\cdot q + A + B_x^T \cdot x + B_y^T \cdot y \} + O^3(x, p, y).
\end{align*}
We have
\begin{equation}
\label{main_eq}
\begin{aligned}
& f - \langle \alpha, \phi \rangle + \{ N - g, F + v\cdot q\} = - \langle \omega, v \rangle \\
& + a - \pw A \\
 & + x^T \left[ b_x - \Gamma  (v + \partial_q A ) + \beta  B_y - \pw B_x \right] \\
 & + y^T  \left[ b_y - B_x - \pw B_y\right]\\
 & + p^T  \left[ b_p - M  (v + \partial_q A) + \Gamma^T  B_y\ -\pw B_p \right] \\
 & + x^T  \left[ d_{xx} + P_{xx} - \Gamma  \partial_q B_x + \beta  D_{xy} - \pw D_{xx} \right]  x\\
 & + y^T  \left[ d_{yy} + P_{yy} -  D_{xy} - \pw D_{yy} \right]  y \\
 & + x^T  \left[ d_{xy} + P_{xy} - \Gamma  \partial_q B_y -   D_{xx} + \beta  D_{yy} - \pw D_{xy}\right]  y \\
 & + p^T  \left[ d_{px} + P_{px} - M  \partial_q B_x - \Gamma  \partial_q B_p + D_{py}  \beta - \pw D_{px} \right]  x \\
 & + p^T  \left[ d_{py} + P_{py} - M  \partial_q B_y - \Gamma^T  D_{yy} - D_{px}   - \pw D_{py} \right]  y \\ 
 & + p^T  \left[ d_{pp} + P_{pp} - M  \partial_q B_p - \Gamma^T  D_{py} - \pw D_{pp}\right]  p \\
 & + \overline{h},
\end{aligned}
\end{equation}
where $\overline{h} = O^3(x, p, y)$. We will define, in an orderly fashion, $A$, $(B_x, B_y - \Avg(B_y))$, $B_p, (\alpha, v, \Avg(B_y))$, $(D_{xx}, D_{yy}, D_{xy})$, $(D_{px}, D_{py}), D_{pp},$ so that for each line in equation (\ref{main_eq}) the expression inside the square parentheses is equal to its average over $\td$. The parentheses in the list indicate that the functions are chosen simultaneously. 

Furthermore, we chose $\alpha, v,$ as well as some of the averages of the functions in the defition of $F$, so that the coefficients of the monomials $x, y, p, xy, py, yy$ in the RHS of equation (\ref{main_eq}) are equal to zero. 
 Let
\[U = \{ \varphi \in \T^l \,\mid\, \nu_{\max}(\beta(\varphi)) < \delta \}.\]
We start by defining the functions in the open set $U$. Later we will extend these functions to $\T^l$ with the help of a smooth bump function.  Most of the definitions of these functions will be obtained as solutions of certain small divisors problems which are considered in Section \ref{sc: lemmata}.  For this, it is essential to notice that for $\varphi \in U$
\[ \nu_{\max}(\beta(\varphi) ) \leq \frac{1}{4}\min_{|k| \leq K} \langle \omega, k\rangle^2,\]
which guarantees that the hypotheses of Lemmas \ref{coupled_equation_1}, \ref{coupled_equation_2} are verified for the matrix $\beta(\varphi)$. 

In the following we fix $\varphi \in \T^l$ and suppose all the functions are evaluated at $\varphi$. To simplify the notation we will omit the evaluation point. Also, to avoid the double subscript, we denote $\| \cdot \|_{\T^r}$ simply by $\| \cdot \|_{r}.$

 Assume $\alpha, v, \Avg(B_y)$ have been defined and let us define $A, B_x, B_y - \Avg(B_y), B_p$. As we shall see the definitions of $\alpha, v, \Avg(B_y)$ will only depend on $f$ and $\phi$ but we prefer, for clarity of exposition, to postpone their exact definition. 

\begin{itemize}
 \item \underline{\textit{Definition of} $A$}. By Lemma \ref{classic_cohom} 
 \[A  = \firstop(a)\]
  belongs to $ C^\omega(\td_{r - \sigma})$ and satisfies
\[ \left\{ \begin{array}{l}
 \| A\|_{r - \sigma} \precdot \frac{\max\{\epsilon, |\alpha| \} }{\sigma^{\tau + d}}, \\
 a - \partial_\omega A = \Avg(a).
 \end{array} \right. \]
 Notice that 
 \[ A = \firstop(a^f) + \firstop(a^\phi)  \cdot \alpha,\]
 where we denote by $\firstop(a^\phi)$ the operator $\firstop$ applied to each coordinate of $a^\phi$. Thus we can write 
 \[A = A^f + A^\phi \cdot \alpha,\]
 where $A^f$ and $A^\phi$ depend only on $f$ and $\phi$ respectively. 
\item \underline{\textit{Definition of} $B_x, B_y - \Avg(B_y)$}. By Cauchy's estimates 
\[ \| \partial_q A\|_{r - 2\sigma} \precdot \frac{\max\{\epsilon, |\alpha| \}}{\sigma^{\tau+d + 1}}.\] 
 By Lemma \ref{coupled_equation_1}
\[ (B_x, B_y - \Avg(B_y)) = \secondop(b_x - \Gamma  \partial_q A, b_y) \]
belongs to $C^\omega(\td_{r - 3\sigma})^{2l}$ and satisfy
\[ \left\{ \begin{array}{l}
 \| B_x\|_{r - 3\sigma}, \| B_y - \Avg(B_y)\|_{r - 3\sigma} \precdot \frac{\max\{\epsilon, |\alpha| \}}{\sigma^{(2l + 1)\tau + 2d + 1}},\\
 b_x - \Gamma  (v + \partial_q A ) + \beta  B_y - \pw B_x = \Avg(b_x) - \Gamma  v + \beta \Avg(B_y),\\
 b_y - B_x - \pw B_y = 0.
 \end{array} \right.
 \] 
 As before, noticing that
 \[
(B_x, B_y - \Avg(B_y))  = \secondop(b^f_x - \Gamma  \partial_q \firstop(a^f), b^f_y) + \secondop(b_x - \Gamma  \partial_q \firstop(a^\phi), b^\phi_y)  \alpha, \]
 we can write $$(B_x, B_y - \Avg(B_y)) = (B_x^f, B_y^f) + (B_x^\phi, B_y^\phi) \alpha,$$  where $ (B_x^f, B_y^f)$ and $(B_x^\phi, B_y^\phi)$ depend only on $f$ and $\phi$ respectively. 
 \item \underline{\textit{Definition of} $B_p$}. By Lemma \ref{classic_cohom} 
\[ B_p  = \firstop(b_p - M  \partial_q A + \Gamma^T  (B_y - \Avg(B_y) ) \]
belongs to $ C^\omega(\td_{r - 4\sigma})^{d}$ and satisfies
\[ \left\{ \begin{array}{l}
 \| B_p\|_{r - 4\sigma} \precdot \frac{\max\{\epsilon, |\alpha| \}}{\sigma^{(2l + 2)\tau + 3d + 1}}, \\
 b_p - M  (v + \partial_q A) + \Gamma^T  B_y\ -\pw B_p = \Avg(b_p) - M  v + \Gamma^T  \Avg(B_y).
 \end{array} \right. \]
 As for the previous functions, noticing that
\[ B_p  = \firstop(b_p^f - M  \partial_q \firstop(a^f) + \Gamma^T  B^f_y)  + \firstop(b_p^\phi - M  \partial_q \firstop(a^\phi) + \Gamma^T  B^\phi_y)  \alpha,\]
we can write 
\[B_p = B_p^f + B_p^\phi \alpha,\]
  where $B_p^f$ and $B_p^\phi$ depend only on $f$ and $\phi$ respectively. 

 Once we established the dependence of $A, B_x, B_y, B_p$ as functions of $\alpha, v,$ $\Avg(B_y)$ we can find an expression for the latter ones as follows.
 \item \underline{\textit{Definition of} $\alpha, v, \Avg(B_y)$}. The third and fifth lines of equation (\ref{main_eq}) are zero if and only if 
 \begin{equation}
 \label{alpha_v_conditions}
 \left\{ \begin{array}{l} 
 \Avg(b_x^\phi)  \alpha - \Gamma  v + \beta  \Avg(B_y) = - \Avg(b_x^f), \\
 \Avg(b_p^\phi)  \alpha - M  v + \Gamma ^T  \Avg(B_y) = - \Avg(b_p^f).
 \end{array}\right.
 \end{equation}
 Let us denote 
 \begin{gather*}
 G^f = A^f + B^f_x \cdot x + B^f_p \cdot p + B^f_y \cdot y,\\
 G^\phi = A^\phi + B^\phi_x \cdot x + B^\phi_p \cdot p + B^\phi_y \cdot y.
 \end{gather*}
Notice that $G^f$ depends only on $f$ and satisfies
\[\| G^f \|_{r - 4\sigma} \precdot \frac{\epsilon}{\sigma^{(2l + 2)\tau + 3d + 1}}.\]
By definition of $F$ and evaluating at $p = 0,$ $x = 0 = y$ we have
 \begin{align*}
 \Avg(\phi + \{ \phi, F + v \cdot q\}) & = \Avg(\phi) + \Avg(B_y) - \Avg(b_p^\phi) \cdot v \\
 & \quad + \Avg( \nabla (\phi - x) \cdot (X_{G^f} + X_{G^\phi} \cdot \alpha) .
 \end{align*}
The RHS of last equation is equal to zero if 
\begin{equation}
\label{B_y_condition}
\begin{aligned}
 - \Avg(\phi + \nabla (\phi - x) \cdot X_{G^f}) = \Avg( \nabla (\phi - x) X_{G^\phi} ) \cdot \alpha \\
 - \Avg(b_p^\phi) \cdot v + \Avg(B_y). 
 \end{aligned}
 \end{equation}
Equations (\ref{alpha_v_conditions}), (\ref{B_y_condition}) give raise to the following linear system
\[
\begin{array}{rc}
\begin{pmatrix}
\Avg(b_x^\phi) & - \Gamma & \beta \\
 \Avg(b_p^\phi) & - M & \Gamma ^T \\
 \Avg( \nabla (\phi - x) \cdot X_{G^\phi} ) & - \Avg(b_p^\phi) & I_l 
\end{pmatrix} 
\begin{pmatrix}
\alpha \\
v \\
\Avg(B_y) 
\end{pmatrix}  
& \\

= \begin{pmatrix}
\Avg(b_x^f) \\
 \Avg(b_p^f) \\
 \Avg(\phi + \nabla (\phi - x) \cdot X_{G^f}) 
\end{pmatrix} & \\
\end{array}
\]
This system possesses a unique solution since by hypothesis the matrix on the LHS is $2\lambda$ close to the matrix
\[ \begin{pmatrix}
I _l & 0 & 0 \\
0 & - M & 0 \\
0 & 0 & I_l 
\end{pmatrix},\]
and is thus invertible for $\lambda$ sufficiently small. Define $\alpha, v, \Avg(B_y)$ to be the unique solution to the system. Notice that this does not depend on the definitions of $A,$ $B_x,$ $B_p,$ $B_y - \Avg(B_y)$. With these definitions
\[ |\alpha|, |v|, |\Avg(B_y)| \precdot \frac{\epsilon}{\sigma^{(2l + 2)\tau + 3d + 2}}.\] 
\item \underline{\textit{Bound for} $P$.} As $P$ depends only on $v, A, B_x, B_y$, by Cauchy's estimates
\begin{align*}
 \| P\|_{r - 4\sigma} & \precdot \frac{\epsilon}{\sigma^{(2l + 2)\tau + 3d + 2}}.
 \end{align*}
 Since $P$ is quadratic on the variables $x, p, y$ 
\[ \| \nabla^2_{x, p, y} P \|_{r - 4\sigma} \precdot \frac{\epsilon}{\sigma^{(2l + 2)\tau + 3d + 2}}. \]
 \item \underline{\textit{Definition of} $D_{xx}, D_{yy}, D_{xy}$}. Consider now the coupled system given by lines six to eight in the equation. By Cauchy's estimates 
\[ \| \partial_q B_x\|_{r - 4\sigma}, \| \partial_q B_y\|_{r - 4\sigma} \precdot \frac{\epsilon}{\sigma^{(2l + 1)\tau + 2d + 2}}.\] 
 By Lemma \ref{coupled_equation_2} there exist $D_{xx}, D_{yy}, D_{xy} \in M_l(\R)$ obeying 
 \[ \| D_{xx}\|_{r - 5\sigma}, \| D_{xy}\|_{r - 5\sigma}, \| D_{yy}\|_{r - 5\sigma} \precdot \frac{\epsilon}{\sigma^{(5l + 2)\tau + 4d + 2}},\]
 such that the sixth line of equation (\ref{main_eq}) does not depend on $q$ while its seventh and eighth lines are equal to zero. 
 \item \underline{\textit{Definition of} $D_{px}, D_{py}$}. For $1 \leq i \leq l$ fixed, the equations concerning $p_ix,$ $p_iy$ give raise to a coupled system equivalent to that of lines three and four. By Cauchy's estimates
 \[ \| \partial_q B_p\|_{r - 5\sigma} \precdot \frac{\epsilon}{\sigma^{(2l + 2)\tau + 3d + 2}}. \] Thus by Lemma \ref{coupled_equation_1} there exist $D_{px}, D_{px} 
\in M_{d \times l }(\R)$ obeying 
 \[ \| D_{px}\|_{r - 7\sigma}, \| D_{py}\|_{r - 7\sigma}, \precdot \frac{\epsilon}{\sigma^{(7l + 2)\tau + 5d + 2}},\]
 such that the ninth line does not depend on $q$ and the tenth line is equal to zero. 
 \item \underline{\textit{Definition of} $D_{pp}$} By Lemma \ref{classic_cohom} there exist $D_{pp} \in M_{d}(\R)$ obeying 
 \[ \| D_{pp}\|_{r - 8\sigma} \precdot \frac{\epsilon}{\sigma^{(7l + 3)\tau + 5d + 2}},\]
 such that the eleventh line does not depend on $q$. 
\end{itemize}
Summarizing, we just defined $F \in C^{\infty, \omega}(U \times D_{r - 8 \sigma, s}, \C),$ $\alpha \in C^\infty(U, \R^l),$ $v \in C^\infty(U, \R^{d})$
with $F$ of the form 
\begin{equation*}
F = A(\varphi, q) + B^T(\varphi, q) \cdot \left( \begin{array}{ccc} x \\ p \\ y \\ \end{array} \right) + \dfrac{1}{2}\left\langle D(\varphi, q) \left(\begin{array}{ccc} x \\ p \\ y \\ \end{array} \right) , \left( \begin{array}{ccc} x \\ p \\ y \\ \end{array} \right) \right\rangle, 
\end{equation*}
 so that the LHS of equation (\ref{eq: cohomological_equation}) is equal to 
 \[ \overline{c} + \dfrac{1}{2}\langle \overline{M} p , p \rangle + \langle \overline{\Gamma} \cdot p, x \rangle + \frac{1}{2} \langle \overline{\beta} \cdot x, x \rangle + \overline{h},\]
for some smooth functions $\overline{c}, \overline{\beta}, \overline{\Gamma}, \overline{M}$ on $U \subset \T^l$ and a smooth function $\overline{h} = O^3(x, p, y)$ defined on $U \times D_{r - 8\sigma, s}$. The differentiable dependence on $\varphi$ follows from the explicit definition of the functions given by Lemmas \ref{classic_cohom}, \ref{coupled_equation_1}, \ref{coupled_equation_2}. Before giving explicit bounds for the norm of these functions let us extend its domain to $\T^l$ with the aid of a bump function. 

By Lemma \ref{bump_function}, there exists a $C^\infty$ bump function $\psi: \T^l \rightarrow \R$, obeying 
\[ 
\psi(\varphi) = \left\{ \begin{array}{ccc} 
1 & \text{if} & \nu_{\max}(\beta(\varphi)) < 2\delta_+, \\
0 & \text{if} & \nu_{\max}(\beta(\varphi)) > 3\delta_+, \\
\end{array} \right. 
\]
such that
\[ 0 \leq \psi \leq 1, \hone \| \psi\|_{C^2} \precdot \frac{1}{\delta_+^{l + 2}}.\]
Hence the product $\psi F$, a priori defined only over $U$, can be extended by zero to well defined $C^{\infty, \omega}$ function over $\T^l \times D_{r - 8\sigma, s}$ satisfying
\[ \| \psi F \|_{C^{2, 3}(\T^l \times D_{r - 8\sigma, s})} \precdot \frac{\epsilon}{\delta_+^{l + 2} \sigma^{(7l + 3)\tau + 5d + 2}}.\] 
The same extension for the functions $\alpha, v, \overline{c}, \dots,$ give raise to well defined maps over $\T^l$. As an abuse of notation we keep the same letters to denote the extended functions. Let us define $\overline{g} \in C^{\infty, \omega}(\T^l \times D_{r - 8\sigma, s})$ as 
\begin{align*}
 \overline{g} & = f - \langle \alpha, \phi \rangle + \{ N - g(\mathbf{N}) ,F + v \cdot q \} - \left( \overline{c} + \dfrac{1}{2}\langle \overline{M} p , p \rangle + \langle \overline{\Gamma} \cdot p, x \rangle + \frac{1}{2} \langle \overline{\beta} \cdot x, x \rangle + \overline{h} \right). 
\end{align*}
Notice that $\overline{g}$ is equal to zero for all $\varphi$ satisfying $\nu_{\max}(\overline{\beta})(\varphi) < 2\delta_+.$ Let
\[ \overline{\mathbf{N}} = (0, \overline{c}, \overline{\beta}, \overline{\Gamma},  \overline{M}, 0, \overline{g}, \overline{h}).\]
Clearly $\overline{\mathbf{N}} \in \mathcal{N}^{0, 2 \delta_+}_{r, s - 8\sigma}$ and it satisfies equation (\ref{eq: cohomological_equation}). From the bounds in the construction and on $\psi$ we have
\[ \| \alpha\|_{C^2}, \| v\|_{C^2} \precdot \frac{\epsilon}{\delta_+^{l + 2}\sigma^{(2l + 2)\tau + 3d + 2}},\] 
\[ \| \overline{\mathbf{N}}\|_{\mathcal{N}_{r - 8 \sigma, s}} \precdot \frac{\epsilon}{\delta_+^{l + 2} \sigma^{(7l + 3)\tau + 5d + 2}}.\] 
This completes the proof.
\end{proof}

\section{Lemmata}
\label{sc: lemmata}

Let $\mathcal{T}^K(\T^d_r)$ be the space of \textit{trigonometric polynomials of degree at most $K$} over $\T^d_r$ and let $\mathcal{T}_0^K(\T^d_r)$ be the space of functions in $\mathcal{T}^K(\T^d_r)$ with zero mean value. To avoid the double subscript we will denote the norm $\| \cdot \|_{\T_{r}}$ simply by $\| \cdot\|_r$.  Given $f \in C^\omega(\td_r)$ we denote its \textit{Fourier coefficients} by
\[ \widehat{f}(k) = \frac{1}{(2\pi)^d}\int_{\td} f(q) e^{ -i k \cdot q} dq,\]
 for all $k \in \Z^d$.  We recall that Fourier coefficients of analytic functions satisfy
\begin{equation}
\label{bounds_coefficients}
 |\widehat{f}(k)| \leq \| f\|_r e^{-|k|r}
 \end{equation}
 for all $r \in \N$. Given $K \in \N$ we define the \textit{truncation of $f$ of order $K$} by
\[ T_Kf (q) = \sum_{|k| \leq K} \widehat{f}(k) e^{ i k \cdot q}.\]
Proofs of Lemmas \ref{trig_approx}, \ref{classic_cohom} can be found in \cite{poschel_lecture_2001}. For the sake of completeness we reproduce them here. 

\begin{lem}
\label{trig_approx}
Suppose $f \in C^\omega(\td_r)$, $0 < \sigma < \min\{1, r\}$ and $t > 0$. There exists $C = C(d)$ such that
\[ \| f - T_Kf \|_{r - \sigma} \leq C K^de^{-K\sigma} \| f\|_{r}\]
for all  $K \in \N.$ In particular, if 
\[ K \geq t \frac{|\log \|f\|_r|}{\sigma} > \sqrt{d},\]
then
\[ \| f - T_Kf \|_{r - \sigma} \leq C \frac{\|f\|_r^{t + 1}|\log \|f\|_r|^d}{\sigma^d}.\]
\end{lem}
\begin{proof} By (\ref{bounds_coefficients})
 \begin{align*}
\| f - T_K f \|_{r - \sigma} & \leq \sum_{|k| > K} |\widehat{f}(k)|e^{|k|(r - \sigma)} \\
& \leq \| f \|_r \sum_{|k| > K} e^{-|k|\sigma} \\
& \leq \| f \|_r \sum_{n > K} 4^d n^{d - 1} e^{-n\sigma}.
\end{align*} 
Since 
\[ \sum_{n > K} n^{d - 1} e^{-n\sigma} \leq C K^de^{-\sigma K}, \]
for some constant $C$ depending only on $d$, the result follows. 
\end{proof}
\begin{lem}
\label{classic_cohom}
Suppose $\omega \in \R^d$ is Diophantine of type $(\gamma,\tau)$ and let $0 < \sigma < r$. There exists a bounded linear operator $$\firstop : C^\omega(\td_{r}) \rightarrow C^\omega_0(\td_{r - \sigma}),$$ obeying
\[ \| \firstop \| \leq \frac{C}{\gamma \sigma^{\tau + d}},\]
 for some positive constant $C = C(d, \tau)$, which to every $ v \in C^\omega(\td_r)$ associates the unique solution $u \in C^\omega_0(\td_{r - \sigma})$ of the equation
\begin{equation}
\label{firstop_eq}
 \begin{array}{l} 
 \langle \omega, \partial_Q u \rangle = v - \Avg(v).
 \end{array}
 \end{equation}
Furthermore, if $v$ is real analytic so is $u$. 
\end{lem}
\begin{proof} Let $u \in C^\omega(\td)$. A simple calculation shows that
\[ \langle \omega, \partial_Q u \rangle = \sum_{k \neq 0} \langle \omega, k \rangle \widehat{u}(k) e^{ i k \cdot q}.\]
Thus the unique solution to (\ref{firstop_eq}), if it exists, is given by
\[u(q) = \sum_{k \neq 0} \frac{\widehat{v}(k)}{\langle \omega, k \rangle } e^{ i k \cdot q}. \]
 Let us show that $u \in C^\omega(\td_{r - \sigma})$. We have
 \begin{align*}
 \| u\|_{r - \sigma} & \leq \sum_{k \neq 0}\frac{ |\widehat{v}(k)|}{|\langle \omega, k \rangle| } e^{|k|(r - \sigma) }\\
 & \leq \gamma^{-1} \| v\|_r \sum_{k \neq 0} e^{-|k|\sigma} |k|^\tau \\
 & \leq \gamma^{-1} \| v\|_r \sum_{n > 0} n^{\tau + d}e^{-n\sigma} \\
 & \leq \dfrac{C \| v\|_r}{\gamma \sigma^{\tau + d}},
 \end{align*}
for some constant $C$ depending only on $d, \tau$. By definition of the Fourier coefficients of $u$ it follows that $u$ is real analytic if and only if $v$ is real analytic. 
\end{proof} 

\begin{lem}
\label{coupled_equation_1}
Suppose $\omega \in DC_d(\gamma, \tau)$, $K \in \N$, $0 < \sigma < r$ and $\beta \in M_{l}(\R)$ symmetric obeying $\|\beta\| \leq 1$. If
\begin{equation}
\label{eq: smallness_cond_1}
\nu_{\max}(\beta) \leq \frac{1}{2}\min_{|k| \leq K} \langle \omega, k \rangle^2,
\end{equation}
there exists a bounded linear operator $$\secondop : \polk(\td_r)^{2l} \rightarrow \polk(\td_{r - \sigma})^l \times \polkz(\td_{r - \sigma})^l,$$ obeying
\[ \| \secondop \| \leq \frac{C}{\gamma^{2l} \sigma^{2l\tau + d}},\]
 for some positive constant $C$ depending continuously on $d, l, \tau, m, \beta,$ which to every $b_x, b_y \in \polk(\td_r)^l$ associates the unique solution $$(B_x, B_y) \in \polk(\td_{r - \sigma})^l \times \polkz(\td_{r - \sigma})^l$$ of the system 
\begin{equation}
\label{secondop_eq}
\left\{ \begin{array}{l}
\pw B_x - \beta B_y = b_x - \Avg(b_x), \\
\pw B_y + B_x= b_y. \\
\end{array} \right.
\end{equation}
\end{lem}

\begin{proof} Let $k \in \Z^d \setminus \{ 0 \}$. Equation (\ref{secondop_eq}) defines the following linear system for the $k$-th Fourier coefficients of the functions
\[\begin{pmatrix}
 i \langle \omega , k \rangle I_l & -\beta \\
I_l & i \langle \omega , k \rangle I_l
\end{pmatrix} 
\begin {pmatrix}
\widehat{B}_x(k) \\
\widehat{B}_y(k)
\end{pmatrix} = 
\begin {pmatrix}
\widehat{b}_x (k) \\
\widehat{b}_y (k)
\end{pmatrix}.\]
Denote by $M_k$ the LHS matrix. Notice that 
\[ \det (M_k)  = \det( \beta- \langle \omega , k \rangle ^2 I_l ).\]
By (\ref{eq: smallness_cond_1}), if $\lambda \in \sigma(\beta- \langle \omega , k \rangle ^2 I_l)$ then $\lambda \leq -\frac{1}{2}\langle \omega , k \rangle ^2.$ Therefore
\[ |\det(M_k)| \geq 2^{-l}\langle \omega, k \rangle^{2l}.\]
Hence, $M_k$ is invertible and satisfies
\[ \| M_k^{-1}\| \leq \frac{C}{\langle \omega, k \rangle^{2l}}, \]
for some constant $C$ depending only on $l$. The solution to the initial system is given by
\[ \begin {pmatrix}
B_x(q) \\
B_y(q)
\end{pmatrix} = \begin {pmatrix} \Avg (b_y) \\ 0 \end{pmatrix} + \sum_{ 0 < |k| \leq K} e^{i k \cdot q} M_k^{-1} \begin {pmatrix} \widehat{b}_x(k) \\ \widehat{b}_y(k) \end{pmatrix}.\]
Thus, by calculations similar to those of Lemma \ref{classic_cohom}
\[ \|B_x\|_{r -\sigma}, \| B_y\|_{r - \sigma} \leq C\frac{\max\{ \|b_x\|_r, \|b_x\|_r \}}{\sigma^{2l\tau + d}},\]
for some constant $C$ depending only on $d, l, \gamma, \tau$. 
\end{proof}
\begin{lem}
\label{coupled_equation_2}
Suppose $\omega \in DC_d(\gamma, \tau)$, $K \in \N$, $0 < \sigma < r$ and $\beta \in M_{l}(\R)$ symmetric obeying $\|\beta\| \leq 1$. If
\begin{equation}
\label{eq: smallness_cond_2}
\nu_{\max}(\beta) \leq \frac{1}{4}\min_{|k| \leq K} \langle \omega, k \rangle^2,
\end{equation}
there exists a bounded linear operator
\[\thirdop : M_{l \times l}(\polk(\td_r))^{3} \rightarrow M_{l \times l}(\polk(\td_{r - \sigma}))^{3},\]
 obeying
\[ \| \thirdop \| \leq \frac{C}{\gamma^{3l} \sigma^{3l\tau + d}},\]
 for some positive constant $C$ depending only on $d, l, \tau$, which to every $d_{xx},$ $d_{yy},$ $d_{xy} \in M_{l \times l}(\polk(\td_r))$ associates the unique solution $D_{xx},$ $D_{yy},$ $D_{xy} \in M_{l \times l}(\polk(\td_{r - \sigma}))$ of the system
 \begin{equation}
\left\{ \begin{array}{l}
\pw D_{xx} - \beta D_{xy} = d_{xx} - \Avg(d_{xx}), \\
\pw D_{yy} + D_{xy} = d_{yy}, \\
\pw D_{xy} - \beta D_{yy} + D_{xx} = d_{xy}. \\
\end{array} \right.
\end{equation}
\end{lem}

\begin{proof} Given $k \in \Z^d \,\setminus\, \{ 0 \}$ the system implies that the $k$-th Fourier coefficient of the functions must obey 
\begin{equation}
 \begin{pmatrix}
 i \langle \omega , k \rangle I_l & 0 & -\beta \\
0 & i \langle \omega , k \rangle I_l & I_l\\ 
I_l & -\beta & i \langle \omega , k \rangle I_l
\end{pmatrix} \begin {pmatrix}
\widehat{D}_{xx}(k) \\
\widehat{D}_{yy}(k) \\
\widehat{D}_{xy}(k)
\end{pmatrix} = 
\begin {pmatrix}
\widehat{d}_{xx}(k) \\
\widehat{d}_{yy}(k) \\
\widehat{d}_{xy}(k) 
\end{pmatrix}.
\end{equation}
Denote by $M_k$ the LHS matrix. Using elementary matrix operations we can transform $M_k$ into 
\[
\begin{pmatrix}
 i \langle \omega, k \rangle I_l & 0 & 0 \\
0 & i \langle \omega, k \rangle I_l & 0 \\ 
I_l & \beta & i \langle \omega, k \rangle I_l + \frac{2}{i\langle \omega, k \rangle} \beta
\end{pmatrix}.
\]
Notice that 
\[ \det(M_k) = \det(2\beta - \langle \omega, k \rangle^2 I_l).\]
By (\ref{eq: smallness_cond_2}), if $\lambda \in \sigma(2\beta - \langle \omega , k \rangle ^2 I_l)$ then $\lambda \leq -\frac{1}{2}\langle \omega , k \rangle ^2.$ Therefore
\[  |\det(M_k)| \geq 4^{-l} |\langle \omega, k \rangle|^{3l}. \]
Hence, the matrix $M_k$ is invertible and satisfies
\[ \| M_k^{-1}\| \leq \frac{C}{|\langle \omega, k \rangle|^{3l}}, \]
for some constant $C$ depending only on $l$. The solution to the initial system is given by
\[ \begin {pmatrix}
D_{xx}(q) \\
D_{yy}(q) \\
D_{xy}(q) \\
\end{pmatrix} = \begin {pmatrix}
\Avg({d}_{xy}) \\
0 \\
\Avg({d}_{yy}) \\
\end{pmatrix} + \sum_{ 0 < |k| \leq K} e^{i k \cdot q} M_k^{-1} \begin {pmatrix}
\widehat{d}_{xx}(k) \\
\widehat{d}_{yy}(k) \\
\widehat{d}_{xy}(k)
\end{pmatrix}.\]
Thus, by calculations similar to those of Lemma \ref{classic_cohom} 
\[ \|D_{xx}\|_{r -\sigma}, \| D_{yy}\|_{r - \sigma}, \| D_{xy}\|_{r - \sigma} \leq C\frac{\max\{ \|d_{xx}\|_r, \|d_{yy}\|_r, \|d_{xy}\|_r \}}{\gamma^{3l}\sigma^{3l\tau + d}},\]
for some constant $C$ depending only on $d, l, \gamma, \tau$. 
\end{proof}

A simple regularization of the distance function yields to the following.

\begin{lem}
\label{bump_function}
Given $t_2 > t_1$ and $\beta: \td \rightarrow \R$ of class $C^1$ obeying $\| \beta \|_{C^1} \leq 1$ there exists a $C^\infty$ bump function $\psi: \td \rightarrow \R$ such that $0 \leq \psi \leq 1$ and 
\[ 
\psi(\varphi) = \left\{ \begin{array}{ccc} 
1 & \text{if} & \beta(\varphi) < t_1, \\
0 & \text{if} & \beta(\varphi) > t_2. \\
\end{array} \right.
\]
Furthermore for any $r \geq 1$ there exists a constant $C_{r, d}$ depending only on $r, d$ such that
\[ \| \psi \|_{C^r} < C_{r, d} \max\left\{ 1, \dfrac{1}{(t_2 - t_1)^{r + d} }\right\}. \]
\end{lem}
\begin{proof} Let \[ a = \frac{t_2 - t_1}{4}, \hone U = \beta^{-1}((-\infty, t_1 + a)). \]
Define $h :\T^l \rightarrow \R$ as
\[ h(\varphi) = \max\left\{ 1 - d(\varphi, U), 0\right\},\]
where $d(\varphi, U)$ denotes the distance of $\varphi$ to $U$. Consider $\Phi : \R^d \rightarrow \R$ given by 
\[ \Phi(x) = \left\{ \begin{array}{cl}
c \exp\left(\frac{-1}{1 - \| x\|^2}\right) & \text{ if } \| x\| < 1, \\
0 & \text{ else, } 
\end{array} \right.\]
where $c = c(d)$ is chosen so that $\int_{\R^d} \Phi(x)dx = 1$. Let 
\[ \Phi_a(x) = \frac{1}{a^d}\Phi\left(\frac{x}{a}\right).\]
Clearly \[\| \Phi_a\|_{C^r} = \frac{1}{a^{r + d}}\| \Phi\|_{C^r}.\]
Define $\psi : \T^l \rightarrow \R$ by $ \psi = \Phi \ast h.$ Notice that 
\[ \psi(\varphi )= \int_{\R^d}\Phi(x) h (\varphi + x)dx = \int_{B_a(0)}\Phi(x) h (\varphi + x)dx. \]
Since $\| \beta \|_{C^1} \leq 1$ it is simple to check that $\psi$ satisfies the desired properties. \end{proof} 

\begin{lem}
\label{C2_bounds}
 Let $r, s > 0$, $d, l \in \N$. There exists a constant $C$, depending only on $d$, $l$, such that for any $C^2$ transformations   $\Phi: \T^l \times U \subset \T^l \times D_{r,s} \rightarrow \T^l \times D_{r,s},$ $\Psi: \T^l \times V \subset \T^l \times D_{r,s} \rightarrow \T^l \times U,$
  obeying
\[ \| \Phi - \id \|_{C^2(U \times D_{r,s})} \leq \epsilon_0 \leq C^{-1}, \hone \| \Psi - \id \|_{C^2(V \times D_{r,s})} \leq \epsilon \leq \epsilon_0,\]
the following holds
\[ \| \Phi \circ \Psi \|_{C^2(V \times D_{r,s})} \leq (1 + \epsilon_0)(1 + \epsilon). \]
\end{lem}
\begin{proof} From
\[ D(\Phi \circ \Psi) = D\Phi \circ \Psi \cdot D \Psi, \]
\[D^2(\Phi_i \circ \Psi) =D\Psi^T \cdot D^2 \Phi_i \circ \Psi \cdot D\Psi + D^2 \Psi \cdot D\Phi_i \circ \Psi, \]
it follows easily that there exists a constant $C = C(d,l)$ such that
\begin{align*}
 \| D(\Phi \circ \Psi) \|_{C^2(V \times D_{r,s})} & \leq \max \{ 1 + \epsilon_0, (1 +\epsilon_0)(1 + \epsilon), C \epsilon_0 (1 + \epsilon)^2 + C\epsilon (1+ \epsilon_0) \} \\
 & \leq (1 + \epsilon_0)(1 + \epsilon).
 \end{align*}
 \end{proof}
 
\section{Appendix: Conjugation by 1-time maps}
\label{conjugation}
Let $(M^{2d}, \omega)$ be a symplectic manifold. Functions in $C^\infty(M)$ are called \textit{Hamiltonians}. To every Hamiltonian $h$ we associate a vector field $X_h$, defined as the unique smooth vector field obeying
\[ i_{X_h}\omega(\cdot) = d_h(\cdot),\]
where
\[ i_{X_h}\omega(v_p) = \omega(X(p), v_p),\]
for all $p \in M$, $v_p \in T_pM$. We denote the flow associated to $X_h$ by $\Psi^t_h$. 
In local conjugated coordinates $(q, p)$ the vector field $X_h$ is given by
\[ X_h = \begin{pmatrix} \partial_p h \\ - \partial_q h \end{pmatrix}.\]
Given $g \in C^\infty(M)$ a direct calculation leads to
\begin{equation}
\label{flow_derivative}
 \dfrac{d}{dt}g \circ \Psi^t = \lbrace g, h \rbrace \circ \Psi^t,
 \end{equation}
where $\{ g, h\}$ denotes the \textit{Poisson bracket}. In local conjugated coordinates $(q, p)$ the Poisson bracket of two functions is given by 
\[ \{ g, h\} = \sum_{i = 1}^d \partial_{q_i}g \partial_{p_i}h - \partial_{p_i}g \partial_{q_i}h.\]
In coordinate-free language we can define the Poisson bracket as
\[ \{ g, h\} = \omega(X_g, X_h).\]
We extend this notation to vector valued functions $\mathbf{g} = (g_1, \dots, g_l) \in C^\infty(M)^l$ as
\[ \{ \mathbf{g}, h\} = (\{g_1, h\}, \dots, \{g_l, h\}).\]
Let $g_0 = g$ and define recursively 
\[ g_n = \dfrac{1}{n} \lbrace g_{n - 1}, h \rbrace,\]
for all $n \geq 1$. By (\ref{flow_derivative}), given $K \in \N$
\[ g \circ \Psi^t_h = \sum_{n = 0}^K t^n g_n + o(t^K).\]
In particular, assuming $\Psi^1_h$ is well defined 
\begin{equation} 
\label{eq: composition_formula}
\begin{aligned}
g \circ \Psi^1_h & = g + \int_0^1 \lbrace g,h \rbrace \circ \Psi^t_h dt \\
& = g + \lbrace g, h \rbrace + \int_0^1 (1-t) \lbrace \lbrace g,h \rbrace ,h \rbrace \circ \Psi^t_h dt.
\end{aligned}
\end{equation}
As an abuse of notation for $M = D_{r,s}$ we extend the definition of Hamiltonian vector fields to expressions of the form $h + v \cdot q$ by setting 
\[ X_{h + v \cdot q} = \begin{pmatrix} \partial_p h \\ - \partial_q h - v \end{pmatrix}.\]
Notice that $h + v \cdot q$ is not a well defined function over $D_{r,s}$ but nevertheless its gradient (and therefore its Poisson bracket with any other function) is well defined. 
\begin{lem}
\label{composition_bounds}
Let $r , s > \sigma > 0$, $d, l \in \N$, $F \in C^{k_1, k_2+1}(\T^l \times D_{r, s}, \C)$, $v \in C^{k_1}(\T^l, \C^{d - l})$. Denote by $\Psi^t$ the Hamiltonian flow associated to $F + v \cdot q$ and let
\[ \epsilon = \| \nabla(F + v \cdot q) \|_{C^{k_1,k_2}(\T^l \times D_{r, s})}. \]
There exists $C = C(l, d, k_1, k_2)$ such that the transformation $\Psi^t : \T^l \times D_{r - \sigma, s - \sigma} \rightarrow D_{r, s}$ is well defined for all 
\[ |t| \leq \frac{\min\{1, \sigma\} }{\epsilon + C\epsilon^2},\]
 and satisfies
\[ \| \Psi^t - id \|_{C^{k_1,k_2}(\T^l \times D_{r - \sigma, s - \sigma})} \leq t (\epsilon + C\epsilon^2). \] 
\end{lem}
\begin{proof} We prove the lemma only for $t \geq 0$ since the case $t \leq 0$ is analogous. Let 
 \[T = \sup \left\{ t \geq 0 \,\left|\, \begin{array}{l} \Psi^t : \T^l \times D_{r - \sigma, s - \sigma} \rightarrow D_{r, s} \text{ is well defined,} \\ \| \Psi^t - id \|_{C^{k_1,k_2}(\T^l \times D_{r - \sigma, s - \sigma})} \leq \min\{1, \sigma\}.\end{array} \right.\right\}.\]
Then, for all $0 \leq t < T$ 
\begin{equation}
\label{CompositionFormula}
\Psi^t = id + t X_{F + v \cdot q} + \int_0^{t} (1 - s) \{X_{F + v \cdot q}, F + v \cdot q \} \circ \Psi^s ds .
 \end{equation} 
Last equation yields to
\[ \| \Psi^t - id\|_{C^{k_1,k_2}(\T^l \times D_{r - \sigma, s - \sigma})} \leq t (\epsilon + C \epsilon^2), \]
for some constant $C$ depending only on $d, l, k_1, k_2$. Since either $T = +\infty$ or 
\[ \limsup_{t \rightarrow T^-} \| \Psi^t - id\|_{C^{k_1,k_2}(\T^l \times D_{r - \sigma, s - \sigma})} \geq \min\{1, \sigma\},\]
it follows that
\[T \geq \frac{\min\{1,\sigma\}}{\epsilon + C \epsilon^2}.\]
\end{proof}

\section{Acknowledgments}

I would like to thank Håkan Eliasson and Bassam Fayad for their constant support during the realisation of this work. I would also like to thank Raphaël Krikorian for several fruitful discussions.

\bibliographystyle{acm}
\bibliography{Resonances.bib}

\end{document}